\documentclass{article}
\usepackage[utf8]{inputenc}
\usepackage[top=30mm, left=30mm, right=30mm, bottom=30mm]{geometry}
\usepackage{amsmath,amssymb,amsthm}
\usepackage{dsfont}
\usepackage{color}
\usepackage{lineno,hyperref}

\modulolinenumbers[5]

\newtheorem{theorem}{Theorem}[section]
\newtheorem{lem}[theorem]{Lemma}

\newtheorem{kor}[theorem]{Corollary}
\newtheorem{prop}[theorem]{Proposition}

\theoremstyle{definition}
\newtheorem{dfn}[theorem]{Definition}
\newtheorem{rem}[theorem]{Remark}

\begin{document}

\title{On Cherny's results in infinite dimensions:
A theorem dual to Yamada-Watanabe}

\author{Marco Rehmeier\footnote{Faculty of Mathematics, Bielefeld University, 33615 Bielefeld, Germany. E-Mail: mrehmeier@math.uni-bielefeld.de }}

\date{}
\maketitle
 
\begin{abstract}

We prove that joint uniqueness in law and the existence of a strong solution imply pathwise uniqueness for variational solutions to stochastic partial differential equations of the form \begin{align*}
\text{d}X_t=b(t,X)\text{d}t+\sigma(t,X)\text{d}W_t, \,\,\,t\geq 0,
\end{align*} and show that for such equations uniqueness in law is equivalent to joint uniqueness in law. Here $W$ is a cylindrical Wiener process in a separable Hilbert space $U$ and the equation is considered in a Gelfand triple $V \subseteq H \subseteq E$, where $H$ is some separable (infinite-dimensional) Hilbert space. This generalizes the corresponding results of A. Cherny for the case of finite-dimensional equations.
\end{abstract}
	\textbf{Keywords:} Stochastic differential equations; Yamada-Watanabe theorem; pathwise uniqueness; uniqueness in law; joint uniqueness in law; variational solutions\\ \\
	\textbf{2010 MSC}: 60H15

\section{Introduction}
The connection between existence and uniqueness of weak and strong solutions is fundamental to the research area of stochastic differential equations. A starting point was the celebrated paper \cite{yamada1971} by Yamada and Watanabe in 1971, in which the authors prove that weak existence and pathwise uniqueness yield the existence of a unique strong solution for finite-dimensional stochastic differential equations. Later several authors worked on a dual statement of this seminal result, i.e. on the implication
 \begin{align*}
 \text{Joint uniqueness in law + existence of strong solution} \Rightarrow \text{pathwise uniqueness}.\tag{Dual}
 \end{align*} A proof of (Dual) can be found in the works of Jacod (\cite{Jacod}) and Engelbert (\cite{Engelbert}). Unfortunately, verifying joint uniqueness in law turns out to be rather difficult in applications. In 2001, Cherny contributed a substantial improvement to this dual result by showing the equivalence of uniqueness in law and joint uniqueness in law for finite-dimensional equations in \cite{ChernyiPaper}. This striking result provides further structural insight into the interplay of the aforementioned notions of existence and uniqueness. \\ \\Recently the study of stochastic partial differential equations, which are necessarily infinite-dimensional equations, attracted much attention and nurtured extensive research activity in this direction. In \cite{RPaper}, Röckner, Schmuland and Zhang extended the classical Yamada-Watanabe theorem to the framework of variational solutions for infinite-dimensional equations in Hilbert spaces. Naturally this brings up two questions, namely  ``Does the dual result (Dual) also hold in this infinite-dimensional framework?'' and ``Can Cherny's result on the equivalence of uniqueness and joint uniqueness in law be generalized to this setting?''.\\ \\ In this paper we give affirmative answers to both question: We prove (Dual) in the framework of the variational approach for solutions to stochastic partial differential equations of the form \begin{align*}
 \text{d}X_t=b(t,X)\text{d}t+\sigma(t,X)\text{d}W_t, \,\,\,t\geq 0,
 \end{align*}in a (infinite-dimensional) Gelfand triple $V \subseteq H \subseteq E$ with a separable Hilbert space $H$, where $W$ is a cylindrical Wiener process in another separable Hilbert space $U$. Further we prove the equivalence of uniqueness and joint uniqueness in law for deterministic initial conditions to such equations. We point out that both statements have also been stated in \cite{Qiao} by Qiao. For a comparison to this work, see Remark \ref{CompQiao}.\\
\\We stress that (Dual) and the equivalence of uniqueness and joint uniqueness in law have also been discussed for other types of equations and notions of solutions: Ondrejat provided affirmative answers to both questions in the setting of mild solutions for Banach space-valued equations in \cite{Ondrejat}. See Remark \ref{Comp Ondrejat} for a more detailed comparison to his work. In \cite{Kurtz}, Kurtz deals with a more general type of stochastic equations and in particular considers (Dual) in this more general framework. However, the equivalence of uniqueness and joint uniqueness in law is not discussed in his setting.\\ \\
 This paper is organized as follows: In the second section we clarify notation and introduce the general framework, including the relevant notions of existence and uniqueness of solutions. The third section contains both main theorems. We present an outline of both proofs in order to render a better understanding of the detailed proofs later on. An explanation on why we have to restrict the second main theorem to deterministic initial conditions is also included. The final section contains the proofs of the main results as well as necessary preparations. Appendix A contains further preparations and, for the convenience of readers, who are not familiar with stochastic integration in detail, Appendix B reviews stochastic integration with respect to Hilbert space-valued martingales, since this will be of great importance within our proofs.
\section{Preliminaries}
\subsection{Notation}
The set of all probability measures on a $\sigma$-algebra $\mathcal{A}$ will be denoted by $\mathcal{M}^+_1(\mathcal{A})$. Given a measure space $(\Omega,\mathcal{F},\mathbb{P})$, the $\sigma$-algebra $\overline{\mathcal{F}}^{\,\mathbb{P}}$ denotes the completion of $\mathcal{F}$ with respect to $\mathbb{P}$. For $I = [0,T]$ for $T > 0$ or $I = \mathbb{R}_+$ we call $(\Omega, \mathcal{F}, (\mathcal{F}_t)_{t \in I}, \mathbb{P})$ a \textit{stochastic basis}, if $\mathcal{F}$ is complete with respect to $\mathbb{P}$ and  $(\mathcal{F}_t)_{t \in I}$ is a right-continuous filtration such that every zero set is contained in $\mathcal{F}_0.$ In this case we denote the corresponding predictable $\sigma$-algebra by $\mathcal{P}_T$ $(\text{if }I = [0,T])$ or $\mathcal{P}_{\infty}$ $(\text{if }I = \mathbb{R}_+)$. We say that a process $X = (X_t)_{t \in I}$ on a stochastic basis is $(\mathcal{F}_t)$-predictable, if $X$ is predictable and we want to stress the dependence on the underlying filtration $(\mathcal{F}_t)_{t \in I}$.\\ \\
Given two separable Hilbert spaces $U$ and $H$, Lin$(U,H)$ denotes the set of linear maps between $U$ and $H$ and $\text{L}(U,H)$ is the subset of all such operators, which are bounded and defined on the whole of $U$. For the adjoint of $A \in \text{L}(U,H)$ we write $A^*$. L$_2(U,H)$ is the set of all Hilbert-Schmidt-operators, i.e. the subset of elements $A$ of $\text{L}(U,H)$ such that $||A||_{\text{L}_2(U,H)} := \big(\sum_{k=1}^{\infty}||Ae_k||^2_{H}\big)^{\frac{1}{2}}<\infty$ for some (hence every) orthonormal basis $(e_k)_{k \in \mathbb{N}}$ of $U$. Equipped with the inner product $(A,B) \mapsto \sum_{k=1}^{\infty}\langle Ae_k,Be_k \rangle_{H}$, $\text{L}_2(U,H)$ becomes a separable Hilbert space. The subset $\text{L}_1(U)$ of $\text{L}(U)$ denotes the set of all nuclear operators on $U$ and $\text{L}^+_1(U)$ is the set of all nuclear operators, which are symmetric and non-negative. Every $A \in \text{L}_1(U)$ has $\textit{finite trace}$ (i.e. $\text{tr}(A) :=\sum_{k=1}^{\infty}\langle A e_k,e_k \rangle _U < +\infty$) and $A \in \text{L}_1^+(U)$ if and only if $A$ is symmetric, non-negative and of finite trace.
\subsection{Basic Setting}
Large parts of the framework presented in this subsection are as in Appendix E of \cite{RSPDE}. Let  ($H,\langle\cdot\,,\cdot\rangle_{H}$) and $(U,\langle \cdot , \cdot \rangle_U)$ be real separable (infinite-dimensional) Hilbert spaces with norms $||\cdot ||_{H}$ and $||\cdot||_U$, respectively. Further let $V$ and $E$ be real separable Banach spaces with norms $||\cdot ||_{V}, ||\cdot ||_{E}$, respectively, such that $V \subseteq H \subseteq E$ continuously and densely. Then Kuratowski's theorem \cite[p.487]{Kuratowski} implies $$V \in \mathcal{B}(H), \,\mathcal{B}(V) = \mathcal{B}(H)\cap V \text{ and } H \in \mathcal{B}(E), \,\mathcal{B}(H) = \mathcal{B}(E)\cap H.$$
For $x \in H$ the map
$$ x \mapsto ||x||_{V} := \begin{cases} 
||x||_V, & x \in V \\
+\infty, & x \in H\backslash V
\end{cases}
$$ is $\mathcal{B}(H)$-measurable and lower semicontinuous on $H$. Thus the path space
$$\mathbb{B} := \bigg\{\omega \in C(\mathbb{R}_+;H)\big{|}\int_{0}^{T}||\omega(s)||_V \text{d}s < \infty \text{ for all } T \geq 0 \bigg\}$$ is well-defined. We define a filtration on $\mathbb{B}$ by $\mathcal{B}_t(\mathbb{B}) := \sigma(\pi_s|0 \leq s \leq t)$ for any $t \geq 0$. Further $(\mathcal{B}_t^+(\mathbb{B}))_{t \geq 0}$ denotes the corresponding right-continuous filtration. Here $\pi_t: \mathbb{B} \to H$ is the canonical projection, i.e. $\pi_t(\omega) = \omega(t)$ for $\omega \in \mathbb{B}$. Note that $(\mathbb{B},\rho)$ is a complete separable metric space, with metric $\rho$ defined through
$$\rho(\omega_1,\omega_2) := \sum_{k=1}^{\infty}2^{-k}\bigg[\bigg(\int_{0}^{k}||\omega_2(s)-\omega_1(s)||_V \text{d}s + \underset{t \in [0,k]}{\text{sup}}||\omega_2(t)-\omega_1(t)||_{H}\bigg)\wedge 1\bigg].$$ \\We denote the Borel $\sigma$-algebra of $(\mathbb{B},\rho)$ by $\mathcal{B}(\mathbb{B})$.
\subsubsection*{The stochastic differential equation under investigation}
We consider stochastic differential equations of the form
\begin{equation}\label{Equation}
\text{d}X_t=b(t,X)\text{d}t+\sigma(t,X)\text{d}W_t, \, \, \,t \geq 0,
\end{equation}which is a formal notation for the integral equation
$
X_t = X_0+ \int_{0}^{t}b(s,X)\text{d}s+\int_{0}^{t}\sigma(s,X)\text{d}W_s, \, \, \, t \geq 0,
$ where the first integral is a pathwise $E$-valued Bochner-integral and the second one is an $H$-valued stochastic Itô-integral. We assume that $b: \mathbb{R}_+ \times \mathbb{B} \to E$ , $\sigma: \mathbb{R}_+ \times \mathbb{B} \to \text{L}_2(U,H)$ and $W = (W_t)_{t \geq 0}$ fulfill the following properties. \\
\\
\textbf{Assumption 1}.
\begin{enumerate}
	\item [(i)] $b$ is $\mathcal{B}(\mathbb{R}_+)\otimes \mathcal{B}(\mathbb{B})/\mathcal{B}(E)$-measurable and $b(t,\cdot)$ is $\mathcal{B}_t(\mathbb{B})/\mathcal{B}(E)$-measurable for all $t \geq 0$,
	\item[(ii)] $\sigma$ is $\mathcal{B}(\mathbb{R}_+)\otimes \mathcal{B}(\mathbb{B})/\mathcal{B}(\text{L}_2(U,H))$-measurable and $\sigma(t,\cdot)$ is $\mathcal{B}_t(\mathbb{B})/\mathcal{B}(\text{L}_2(U,H))$-measurable for all $t \geq 0$,
	\item[(iii)] $W$ is an $(\mathcal{F}_t)$-$\mathbb{R}^{\infty}$-Wiener process on $U$ with covariance $\text{id}_{U}$ on a stochastic basis $(\Omega, \mathcal{F}, (\mathcal{F}_t)_{t \geq 0}, \mathbb{P})$, i.e. formally $W_t = \sum_{k=1}^{\infty}\beta_k(t)e_k,$ where $(e_k)_{k \in \mathbb{N}}$ is an orthonormal basis of $U$ and ($\beta_k)_{k \in \mathbb{N}}$ is a family of independent real-valued $(\mathcal{F}_t)$-Brownian motions on $\Omega$. We also write $W = (\beta_k)_{k \in \mathbb{N}}$ and call $W$ a \textit{standard }$\mathbb{R}^{\infty}$\textit{-Wiener process}. 
	\end{enumerate}
	The stochastic integral in (\ref{Equation}) is defined through
	$\int_{0}^{t}\sigma(s,X)\text{d}W_s := \int_{0}^{t}\sigma(s,X)\circ J^{-1}\text{d}\bar{W}_s.$ Here $J:U\to \bar{U}$ is a one-to-one Hilbert-Schmidt-map with values in a separable Hilbert space $(\bar{U},\langle \cdot, \cdot \rangle _{\bar{U}})$ and $\bar{W}_t := \sum_{k=1}^{\infty}\beta_k(t)Je_k,\,\,\,t \geq 0$, is the \textit{cylindrical Wiener process} associated to $W$. The orthonormal basis $(e_k)_{k \in \mathbb{N}}$ is the same as in Assumption 1 above, which we fix from now on. Such $J$ and $\bar{U}$ always exist and the definition of the stochastic integral does not depend on the choice of $J$ or $\bar{U}$. Further $\bar{W}$ is a $\bar{Q}$-Wiener process with $\bar{Q} := JJ^*$ $\in$ $\text{L}^+_1(\bar{U})$. We fix such $J$, $\bar{U}$ and $\bar{Q}$ from now on. For technical details about stochastic integration with respect to cylindrical Wiener processes we refer to \cite[Section 2.5.]{RSPDE}.\\
	\\The paths of $\bar{W}$ are elements of the space
	$\mathbb{W}_0 := \{\omega \in C(\mathbb{R}_+,\bar{U})\big{|}\omega(0) =0\}.$ Define a metric on $\mathbb{W}_0$ through
	$$\zeta(\omega_1,\omega_2) :=\sum_{k=1}^{\infty}2^{-k}\big(||\omega_1-\omega_2||_{L^{\infty}([0,k];\bar{U})}\wedge 1\big)$$ and observe that $(\mathbb{W}_0,\zeta)$ is a complete separable metric space. We define a filtration on $(\mathbb{W}_0,\mathcal{B}(\mathbb{W}_0))$ through $\mathbb{B}_t(\mathbb{W}_0):= \sigma(\pi_s|0\leq s \leq t)$, where as before $\pi_t$ denotes the canonical projection. Note $\mathcal{B}(\mathbb{W}_0) = \sigma(\pi_t|t \geq 0)$ and that this implies the $\mathcal{F} /\mathcal{B}(\mathbb{W}_0)$-measurability of $\bar{W}:\Omega \to \mathbb{W}_0$, $\omega \mapsto (\bar{W}(\omega)_t)_{t \geq0}$ due to the $(\mathcal{F}_t)$-adaptedness of $(\bar{W}_t)_{t \geq0}$.
\subsubsection*{Strong, weak solutions and notions of uniqueness}

We now present the relevant notions of solutions and uniqueness for our considerations and clarify the relations between them.
\begin{dfn}\label{weak sol}
A pair $(X,W)$ is called a \textit{weak solution} to Eq. (\ref{Equation}), if $X=(X_t)_{t \geq 0}$ is an $(\mathcal{F}_t)$-adapted process with paths in $\mathbb{B}$ and $W$ is a standard $(\mathcal{F}_t)$-$\mathbb{R}^{\infty}$-Wiener process on some stochastic basis $(\Omega, \mathcal{F}, (\mathcal{F}_t)_{t \geq 0}, \mathbb{P})$ such that the following holds true:
\begin{enumerate}
	\item[(i)] $\int_{0}^{T}||b(s,X)||_E\,\text{d}s + \int_{0}^{T}||\sigma(s,X)||^2_{\text{L}_2(U, H)}\text{d}s < +\infty \, \, \, \, \mathbb{P}\text{-a.s.}$ for every $T \geq 0$.
	\item[(ii)] $X_t=X_0 + \int_{0}^{t}b(s,X)\text{d}s + \int_{0}^{t}\sigma(s,X)\text{d}W_s$ for every $t \geq 0$  $\mathbb{P}\text{-a.s.}$ as an equation on $E$.
	\end{enumerate}We call $X$ a \text{solution process} of Eq. (\ref{Equation}) or simply \textit{solution}. Note that such $X$ is $\mathcal{F}/\mathcal{B}(\mathbb{B})$-measurable.
\end{dfn}
\begin{dfn} \label{weak uniqu}\begin{enumerate} \item [(i)]\textit{Weak uniqueness} (also \textit{uniqueness in law}) holds for Eq. (\ref{Equation}), if for any two solutions $(X,W^1)$ and $(Y,W^2)$ on (possibly different) stochastic bases $(\Omega, \mathcal{F}, (\mathcal{F}_t)_{t \geq 0}, \mathbb{P})$ and $(\Omega', \mathcal{F}', (\mathcal{F}'_t)_{t \geq 0}, \mathbb{P}')$, respectively, 
\begin{equation}\label{initial distr equal}
\mathbb{P} \circ X_0^{-1} = \mathbb{P}'\circ Y_0^{-1}
\end{equation}implies $\mathbb{P}\circ X^{-1} = \mathbb{P}'\circ Y^{-1}$ as measures on $(\mathbb{B},\mathcal{B}(\mathbb{B}))$.
\item[(ii)] \textit{Weak uniqueness given }$\mu \in \mathcal{M}^+_1(\mathcal{B}(H))$ holds, if the implication in (i) is at least valid for all weak solutions $(X,W^1), (Y,W^2)$ with initial distribution $\mu$.
					\item[(iii)] Eq. (\ref{Equation}) has \textit{joint uniqueness in law} (also joint weak uniqueness), if in the setting of (i) (\ref{initial distr equal}) implies 
					$\mathbb{P}\circ(X,\bar{W}^1)^{-1} = \mathbb{P}'\circ(Y,\bar{W}^2)^{-1}$ as measures on $\mathcal{B}(\mathbb{B})\otimes \mathcal{B}(\mathbb{W}_0)$. The definition of \textit{joint uniqueness in law given }$\mu \in \mathcal{M}^+_1(\mathcal{B}(H))$ is analogue to (ii).
					\item[(iv)] $\delta$\textit{-weak uniqueness} and $\delta$\textit{-joint weak uniqueness} hold, if the respective implications in (i) and (iii) hold at least when (\ref{initial distr equal}) is restricted to
					\begin{equation*}
					\mathbb{P} \circ X_0^{-1} = \delta_x = \mathbb{P}'\circ Y_0^{-1}
					\end{equation*}for every $x \in H$, i.e. to arbitrary deterministic initial conditions. $\delta_x$ denotes the Dirac-measure in $x$. 
					\end{enumerate}
					\end{dfn}
\begin{dfn}
\begin{enumerate}
\item [(i)] \textit{Pathwise uniqueness} holds for Eq. (\ref{Equation}), if for any two weak solutions $(X,W)$, $(Y,W)$ on a common stochastic basis $(\Omega, \mathcal{F}, (\mathcal{F}_t)_{t \geq 0}, \mathbb{P})$ with a common standard $\mathbb{R}^{\infty}$-Wiener process $W$, 
$X_0 = Y_0 \, \, \,\mathbb{P}\text{-a.s.}$ implies $X_t = Y_t $\text{ for all }$ \,t \geq 0 \, \, \, \mathbb{P}\text{-a.s.}$
\item[(ii)] For $\mu \in \mathcal{M}^+_1(\mathcal{B}(H))$, \textit{pathwise uniqueness given} $\mu$ means that the implication in (i) holds at least for all weak solutions as in (i), which additionally satisfy $\mathbb{P}\circ X_0^{-1} = \mu = \mathbb{P}\circ Y_0^{-1}$.
\item[(iii)] $\delta$\textit{-pathwise uniqueness} holds, if the implication in (i) holds at least for all solutions $X$ and $Y$ with $X_0 =x = Y_0$ for any $x \in H$.
\end{enumerate}
\end{dfn}
In order to define the notion of a strong solution, let $\hat{\mathcal{E}}$ denote the set of all maps $F:H\times \mathbb{W}_0 \to \mathbb{B}$ such that for each $\mu \in \mathcal{M}_1^+(\mathcal{B}(H))$ there is a $\overline{\mathcal{B}(H)\otimes \mathcal{B}(\mathbb{W}_0)}^{\mu\otimes \mathbb{P}^{\bar{Q}}}/\mathcal{B}(\mathbb{B})$-measurable map $F_{\mu}$ such that for $\mu$-a.a. $h \in H$ $$F(h,\omega) = F_{\mu}(h,\omega) \text{ for } \mathbb{P}^{\bar{Q}}\text{-a.a. } \omega \in \mathbb{W}_0$$holds. $\mathbb{P}^{\bar{Q}}$ denotes the distribution of the $\bar{Q}$-Wiener process $\bar{W}$ on $(\mathbb{W}_0, \mathcal{B}(\mathbb{W}_0))$. Obviously each $F_{\mu}$ is uniquely determined up to a $\mu \otimes \mathbb{P}^{\bar{Q}}\text{-zero}$ set.
\begin{dfn}\label{strong sol}
Eq. (\ref{Equation}) has a \textit{strong solution}, if there exists $F \in \hat{\mathcal{E}}$ such that for all $h \in H$, $\omega \mapsto F(h,\omega)$ is $\overline{\mathcal{B}_t(\mathbb{W}_0)}^{\mathbb{P}^{\bar{Q}}}/\mathcal{B}_t(\mathbb{B})$-measurable for every $t \geq 0$ and for every standard $(\mathcal{F}_t)$-$\mathbb{R}^{\infty}$-Wiener process $W$ on any stochastic basis $(\Omega, \mathcal{F}, (\mathcal{F}_t)_{t \geq 0}, \mathbb{P})$ and any $\mathcal{F}_0/\mathcal{B}(H)$-measurable map $\xi: \Omega \to H$, the $\mathbb{B}$-valued process
$X := F_{\mathbb{P}\circ \xi^{-1}}(\xi, \bar{W})$ is such that $(X,W)$ is a weak solution to Eq. (\ref{Equation}) with $X_0 = \xi$ $\mathbb{P}$-a.s. We will conventionally call $F$ \textit{the strong solution}.
\end{dfn}
\section{Main Results}
We present the two main theorems of this paper. We give outlines of their proofs and point out why we have to restrict the second theorem to deterministic initial conditions. The sketch of a simple proof for a very special case of the second theorem is included as well in order to demonstrate the idea we follow for the general version. We assume the framework of the previous section to be in force.
\begin{theorem}\label{Main Theorem I}
	Consider the stochastic evolution equation 
	\begin{equation*}
	X_t = X_0 + \int_{0}^{t}b(s,X)\text{d}s+\int_{0}^{t}\sigma(s,X)\text{d}W_s, \, \, \, t \geq 0,
	\end{equation*}where we assume that $b$, $\sigma$ and $W$ fulfill Assumption 1. If this equation has a strong solution and joint uniqueness in law given $\mu$ holds for some $\mu \in \mathcal{M}^+_1(\mathcal{B}(H))$, then pathwise uniqueness given $\mu$ holds as well. In particular, the existence of a strong solution and joint uniqueness in law imply pathwise uniqueness. 
\end{theorem}

\begin{theorem}\label{UL impl JUL}
	Consider the stochastic evolution equation
	\begin{equation}\label{SDE}
	X_t = X_0 + \int_{0}^{t}b(s,X)\text{d}s+\int_{0}^{t}\sigma(s,X)\text{d}W_s, \, \, \, t \geq 0,
	\end{equation}where we assume that $b$, $\sigma$ and $W$ fulfill Assumption 1. For any $x \in H$, weak uniqueness given $\delta_x$ is equivalent to joint uniqueness in law given $\delta_x$. In particular, $\delta$-uniqueness in law is equivalent to $\delta$-joint uniqueness in law.	
\end{theorem}In particular, we obtain the following corollary, which we interpret as a dual statement to the Yamada-Watanabe theorem.
\begin{kor}\label{Main Cor}
	Assume $b$, $\sigma$ and $W$ fulfill Assumption 1. Then for the stochastic differential equation above the existence of a strong solution and $\delta$-weak uniqueness imply $\delta$-pathwise uniqueness.
\end{kor}
\noindent \textbf{Scheme of proof of Theorem \ref{Main Theorem I}}:\\The proof is similar to the one presented by Cherny for the finite-dimensional case in \cite[Thm. 3.2]{ChernyiPaper}.
Assume there exists a strong solution $F$ and joint uniqueness in law given $\mu$ holds for some $\mu \in \mathcal{M}^+_1(\mathcal{B}(H))$. We want to prove that every weak solution is given by the strong solution $F$. \\ \\The main idea is to consider the regular conditional distribution of $Z$ with respect to a suitable sub-$\sigma$-algebra of $\mathcal{F}$, namely the $\mathbb{P}$-completion of $\sigma(\xi_0,\bar{W})$, which in the proof will be called $\mathcal{G}_0^{\bar{W}}$. We will prove that the regular conditional distribution of $X$ with respect to the same $\sigma$-algebra coincides with that of $Z$. This step will heavily rely on the assumption on joint uniqueness in law given $\mu$. From here the definition of regular conditional distributions will imply 
$\mathbb{E}[g(X)|\mathcal{G}_0^{\bar{W}}] = g(Z)$ for any $\mathbb{R}$-valued, bounded, measurable $g$. By joint uniqueness in law given $\mu$, we will easily derive $g(X) = g(Z)$ for all $g$ as above and from there the result is immediate. The ``in particular"-statement of the theorem then follows directly, because joint uniqueness in law is, by definition, equivalent to joint uniqueness in law given $\mu$ for all $\mu \in \mathcal{M}^+_1(\mathcal{B}(H)$.\\ \\\textbf{Scheme of proof of Theorem \ref{UL impl JUL}}:\\ First of all we would like to point out that the proof would be straightforward, if we assumed the operator $\sigma(t,y)$ to be one-to-one for all $(t,y) \in [0,+\infty[\times \mathbb{B}$. Indeed, in this case $\sigma^{-1}(t,y)$ is well-defined on $\text{Im}\,\sigma(t,y)$ and for a weak solution $(X,W)$ we can, setting $N_t := X_t-X_0-\int_{0}^{t}b(s,X)\text{d}s$ $(= \int_{0}^{t}\sigma(s,X)\text{d}W_s$ $\mathbb{P}$-a.s.) for $t \geq 0$, consider the equation
\begin{align*}
\int_{0}^{t}J \sigma(s,X)^{-1}\text{d}N_s = \int_{0}^{t}J \sigma(s,X)^{-1}\text{d}\big(\int_{0}^{s}\sigma(r,X)\text{d}W_r\big) = \int_{0}^{t}J J^{-1}\text{d}\bar{W}_s = \bar{W}_t.
\end{align*}Here we used Proposition \ref{localize} (ii) from Appendix B to obtain the well-definedness of the second (hence also the first) term and to deduce the equality of the second and third integral. Thus we have expressed the Wiener process as a measurable functional of the solution $X$, which yields the desired statement. Although this simple reasoning does not work in the general case we consider, one will recognize the same idea in our proof below. For the general case we basically  follow the ideas of Theorem 3.1. in \cite{ChernyiPaper} and Theorem 1.6 in \cite{Qiao}. The majority of techniques used for the finite-dimensional case has to be modified for our infinite-dimensional variational approach. \\
\\Fix a deterministic initial condition $x \in H$ for which uniqueness in law given $\delta_x$ holds and for which Eq. (\ref{SDE}) has at least one weak solution. We prove that $\mathbb{P}\circ (X,\bar{W})^{-1}$ is uniquely determined by $\mathbb{P}\circ X^{-1}$ for every weak solution $(X,\bar{W})$ with $X_0 = x$ $\mathbb{P}$-a.s. Since we assume uniqueness in law given $\delta_x$, this implies the desired statement. Roughly speaking, we will express the Wiener process $\bar{W}$ as a functional of $X$ and a process independent of $X$. We will arrange the proof in the following steps:
\begin{enumerate}
	\item [(i)]Let $(X,W)$ be a weak solution on a stochastic basis $(\Omega, \mathcal{F}, (\mathcal{F}_t)_{t \geq 0},\mathbb{P})$ such that $X_0 = x$ $\mathbb{P}$-a.s., let $(\Omega', \mathcal{F}', (\mathcal{F}'_t)_{t \geq 0},\mathbb{P}')$ be a second stochastic basis and $W^1$, $W^2$ two independent $\mathbb{R}^{\infty}$-Wiener processes on it.
	Consider the product space $\tilde{\Omega} := \Omega \times \Omega'$ with $\tilde{\mathbb{P}} := \mathbb{P}\otimes \mathbb{P}'$ and the obvious $\sigma$-algebra and filtration such that we obtain a stochastic basis. We define the processes $\tilde{X}, \tilde{W}, \tilde{W}^1$ and $\tilde{W}^2$ on this product space in an obvious way via projections and check that $(\tilde{X},\tilde{W})$ is also a weak solution subject to the initial condition $ x$.
	\item[(ii)] For a linear subspace $V \subseteq U$, let $\text{pr}_V$ denote the orthogonal projection onto $V$. We define the processes $\phi$ and  $\psi: \mathbb{R}_+ \times \mathbb{B} \to \text{L}(U)$ via
	$\phi(t,y) := \text{pr}_{\text{ker}\sigma(t,y)^{\perp}} \text{ and }\psi(t,y) := \text{pr}_{\text{ker}\sigma(t,y)},$ which we will use to split up the integral 
	$\bar{\tilde{W}}_t = \int_{0}^{t}J \circ \text{id}_{U}\circ J^{-1}\text{d}\bar{\tilde{W}}_s$ later on. We further introduce the processes
	$$\bar{V}^1_t := \int_{0}^{t}J \phi(s,\tilde{X}) \,\text{d}\tilde{W}_s +\int_{0}^{t}J \psi(s,\tilde{X})\,\text{d}\tilde{W}_s^1 \text{ and }\bar{V}^2_t := \int_{0}^{t}J \phi(s,\tilde{X}) \,\text{d}\tilde{W}^2_s+\int_{0}^{t}J  \psi(s,\tilde{X})\,\text{d}\tilde{W}_s$$ and verify that these are independent Wiener processes on $\tilde{\Omega}$, for which we will need a Hilbert space version of Lévy's characterization of Brownian motion. Next we show that the pair $(\tilde{X},V^1)$ is a weak solution to Eq. (\ref{SDE}).
	\item[(iii)] In this crucial step we prove the independence of $\tilde{X}$ and $\bar{V}^2$. We will heavily use Lemma \ref{Long rcd lemma} as well as the assumption on uniqueness in law given $\delta_x$. More precisely, we will even show that $\tilde{X}$ is independent of $\mathcal{\tilde{F}}_0 \,\vee \,\sigma(\bar{V}^2_t|t \geq 0)$.
	\item[(iv)] We introduce the pseudo inverse of the diffusion term $\sigma$, i.e. we define $\chi(t,y): \text{Im}\,\sigma(t,y) \to \text{ker}\,\sigma(t,y)^{\perp}$ through $\chi(t,y) := \sigma(t,y)^{-1}$ for every $(t,y) \in [0,+\infty[\times\mathbb{B}$. Now we can, as mentioned above, split up the Wiener process in the following way:
	\begin{align*}
	\bar{\tilde{W}}_t &= \int_{0}^{t}J \chi(s,\tilde{X})\text{d}\bigg(\tilde{X}_s-x-\int_{0}^{s}b(r,\tilde{X})\text{d}r\bigg)+\int_{0}^{t}J \psi(s,\tilde{X}) J^{-1}\text{d}\bar{V}^2_s.
	\end{align*}Due to Step (iii) we know that $\tilde{X}$ is independent of $\bar{V}^2$. The first summand is a measurable functional of $\tilde{X}$. This will imply the result. The ``in particular''-statement of the theorem is then obvious, because $\delta$-(joint) uniqueness in law is by definition equivalent to (joint) uniqueness in law given $\delta_x$ for all $x \in H$.
\end{enumerate}
\begin{rem} With our techniques, Theorem \ref{UL impl JUL} cannot be generalized to non-deterministic initial conditions. Why is this so? Within the proof of Theorem \ref{UL impl JUL} we crucially use Lemma \ref{Long rcd lemma}, as outlined in Step (iii) above. The main point is to obtain --- using the notation of Lemma \ref{Long rcd lemma} --- that $\mathbb{P}_{\omega} \circ \Pi_1(0)^{-1} = \mathbb{P}\circ X_0^{-1} = \delta_x$ holds for  $\mathbb{P}$-a.a. $\omega \in \Omega$. We achieve this through
	\begin{align*}
	\{0,1\} \ni \mathbb{P}(X_0 \in A) = \mathbb{P}\big((X,\bar{W})\in \{\Pi_1(0) \in A\}\big) = \int_{\Omega}\mathbb{P}_{\omega}\big(\{\Pi_1(0)\in A\}\big)\mathbb{P}(\text{d}\omega),
	\end{align*} (c.f. (\ref{oneeleven})), which implies $\mathbb{P}_{\omega}\big(\{\Pi_1(0)\in A\}\big) = \mathbb{P}(X_0 \in A)$ for $\mathbb{P}$-a.a. $\omega \in \Omega$ for every $A \in \mathcal{B}(H)$.
	Unfortunately we cannot drop the condition $\mathbb{P}(X_0 \in A) \in \{0,1\}$ for each $A \in \mathcal{B}(H)$, because else we could not conclude that the integrand $\omega \mapsto \mathbb{P}_{\omega}\big(\{\Pi_1(0)\in A\}\big)$ is constant $\mathbb{P}$-a.s. for every $A \in \mathcal{B}(H)$. Hence a necessary and sufficient condition is that any weak solution $X$ fulfills $$\mathbb{P}\circ X_0^{-1}(A) \in \{0,1\} \text{ for all }A \in \mathcal{B}(H).$$ Due to the separability of $H$, this is equivalent to $\mathbb{P}\circ X_0^{-1} = \delta_x$ for some $x \in H$.
\end{rem}
\begin{rem}\label{Comp Ondrejat}
	In \cite{Ondrejat}, M. Ondrejat considers, among other statements, the assertions of both Theorem \ref{Main Theorem I} and Theorem \ref{UL impl JUL} in the setting of mild solutions to Banach space-valued stochastic differential equations (c.f. Theorem 1 and Theorem 4 in \cite{Ondrejat}, respectively). To retrieve the type of equations we consider, $X$ needs to be a separable Hilbert space. Further, necessarily $X=X_1$ in order to choose $S_t = \text{id}_X$, which is requisite to obtain our type of equations.  This shows that the situation in \cite{Ondrejat} does not contain our approach via a generalized Gelfand triple. Above that one notices that the drift and diffusion term of his type of equations do not depend on entire solutions paths, but only on its current time value. 
\end{rem}
\begin{rem}\label{CompQiao}In \cite{Qiao} H. Qiao states both main theorems of this paper for the same type of equations and within the same framework. Two rather short proofs are given, which mostly follow the same arguments as in Cherny's proofs in \cite{ChernyiPaper} for the finite-dimensional setting. In doing so, central technical issues arising from the infinite-dimensional framework are not properly adjusted to the proof of Theorem 1.6. in \cite{Qiao}. In particular this includes (assuming the notation of \cite{Qiao}) the proof of the independence of $V^1$ and $V^2$ and the calculation of the covariation of $JV^i$, $i \in \{1,2\}$ (note that the reference Proposition 3.13. given for this argument does not apply to the situation on p.372 in \cite{Qiao}, because the stochastic integrals $\phi(\cdot, \tilde{X}).\tilde{W}$ and $\psi(\cdot,\tilde{X}).\tilde{W}^1$ are not necessarily independent processes). Further the well-definedness of $V^1$ and $V^2$ is not discussed and there is no justification for the computations of stochastic integrals on p.373. The final conclusion of the proof is rather imprecise. Furthermore, the important technical preparations in \cite{Qiao}, namely Lemma 2.2. and Lemma 2.3., seem to rely heavily on arguments presented in \cite{RPaper} (c.f. Lemmas 2.4, 2.5 and 2.6 and the arguments inbetween). However, the situation there is, albeit quite similar in nature, technically a different one. Hence we believe it is valuable to present detailed proofs for these technical preparations as well for the main theorem.\\
	\\Concerning the proof of Theorem 1.7. in \cite{Qiao}, note that the situation considered there is less general then our setting in terms of the definition of a \textit{strong solution}. Below we give a proof considering this more general notion of a strong solution, which is also more precise and detailed.
\end{rem}
\section{Proofs of the main results}
The first subsection contains the main technical preparations for the proofs of our main results, which are presented in the second subsection.
\subsection{Preparations}
As before, let $H$ and $U$ be separable, infinite-dimensional Hilbert spaces. We start by recalling the definition and basic properties of regular conditional distributions, since these will be a key tool within the main proofs below.
\begin{dfn}\label{Def r.c.d.}
	Let $X$ be a random variable on $(\Omega, \mathcal{F},\mathbb{P})$ taking values in a Polish space $(E,\mathcal{B}(E))$ and $\mathcal{G} \subseteq \mathcal{F}$ a sub-$\sigma$-algebra. A family of probability measures $(Q_{\omega})_{\omega \in \Omega}$ on $\mathcal{B}(E)$ is called \textit{regular conditional distribution} (often abbreviated \textit{r.c.d.}) \textit{of} $X$ \textit{with respect to} $\mathcal{G}$, if
	\begin{enumerate}
		\item [(i)] $\omega \mapsto Q_{\omega}(A)$ is $\mathcal{G}$-measurable for each $A \in \mathcal{B}(E),$
		\item[(ii)] $\mathbb{E}[\mathds{1}_D \cdot \mathds{1}_{\{X \in A\}}]=\mathbb{E}[\mathds{1}_D \cdot Q_{\omega}(A)]$ holds for all $D \in \mathcal{G}$ and $ A \in \mathcal{B}(E)$. 
	\end{enumerate}
\end{dfn}
The statements of the following remark are well-known results. Thus we omit their proofs.
\begin{rem}\label{Rem rcd}
	\begin{enumerate}
		\item [(i)] For $A \in \mathcal{B}(E)$, $\omega \mapsto Q_{\omega}(A)$ is a version of $\mathbb{E}[\mathds{1}_{\{X \in A\}}|\mathcal{G}]$ with exception set possibly depending on $A$.
		\item[(ii)] For $X$, $\mathcal{G}$ and $E$ as above a unique regular conditional distribution exists.
		\item[(iii)] If $X$ itself is $\mathcal{G}$-measurable, then $\big(\delta_{X(\omega)}(\cdot)\big)_{\omega \in \Omega}$ is the (unique) r.c.d. of $X$ with respect to $\mathcal{G}$.
		\item [(iv)] Let $h: E \to \mathbb{R}$ be $\mathcal{B}(E)$-measurable. If $h \geq 0$ or $h(X) \in L^1(\Omega, \mathcal{F},\mathbb{P})$, then we have
		$\mathbb{E}_{{Q_{\omega}}}[h] = \mathbb{E}[h(X)|\mathcal{G}](\omega) \, \, \, \mathbb{P}\text{-a.s.},$ where $\mathbb{E}_{Q_{\omega}}[\cdot]$ denotes expectation with respect to $Q_{\omega}$ for fixed $\omega \in \Omega$.
	\end{enumerate}
\end{rem}For the next two lemmatas we fix the following framework. Let $(X,W)$ be a weak solution of Eq. (\ref{Equation}) on a stochastic basis $(\Omega, \mathcal{F},(\mathcal{F}_t)_{t \geq 0},\mathbb{P})$ with initial condition 
$X_0 = x$ $\mathbb{P}\text{-a.s.}$ for some $x \in H$ and let $(\mathbb{P}_{\omega})_{\omega \in \Omega}$ be the regular conditional distribution of the random variable $(X,\bar{W}):\Omega \to \big(\mathbb{B}\times \mathbb{W}_0, \mathcal{B}(\mathbb{B})\otimes \mathcal{B}(\mathbb{W}_0)\big)$ with respect to $\mathcal{F}_0$ (by the remark above such a r.c.d. exists, because $\mathbb{B}\times \mathbb{W}_0$ is a complete separable metric space when equipped with the product metric of $\rho$ and $\zeta$ as introduced in Subsection 1.2). For $\omega \in \Omega$ define a stochastic basis through
\begin{align*}
\bar{\Omega} := \mathbb{B}\times \mathbb{W}_0,\,
\bar{\mathcal{F}}^{\omega} := \overline{\mathcal{B}(\mathbb{B})\otimes \mathcal{B}(\mathbb{W}_0)}^{\mathbb{P}_{\omega}},\,
\bar{\mathcal{F}}^{\omega}_t := \bigcap_{\epsilon > 0}\sigma\big(\mathcal{B}_{t+\epsilon}(\mathbb{B})\otimes \mathcal{B}_{t+\epsilon}(\mathbb{W}_0),\mathcal{N}_{\omega}\big),
\end{align*}where $\mathcal{N}_{\omega} := \{N \in \bar{\mathcal{F}}^{\omega}|\mathbb{P}_{\omega}(N)=0\}$. Further $\Pi_1 : \bar{\Omega} \to \mathbb{B},\, \Pi_2: \bar{\Omega} \to \mathbb{W}_0$ denote the canonical projections on the first and second variable, respectively. Note that, as pointed out in Remark \ref{CompQiao} above, the following two statements are in spirit of Lemmata 2.4, 2.5 and 2.6 in \cite{RPaper} and also that Lemma \ref{Long rcd lemma} below is reminiscent to Lemma 3.3. in \cite{ChernyiPaper}. 
\begin{lem}\label{StilltoCome1}
	Let $\bar{Q}$ and $\bar{U}$ be as in Subsection 1.2. Then $\Pi_2$ is a $\bar{U}$-valued $(\bar{\mathcal{F}}_t^{\omega})$-$\bar{Q}$-Wiener process on $(\bar{\Omega},\bar{\mathcal{F}}^{\omega},(\bar{\mathcal{F}}_t^{\omega})_{t \geq 0},\mathbb{P}_{\omega})$ for $\mathbb{P}$-a.a. $\omega \in \Omega$.
\end{lem}
\begin{proof}
	Since $\Pi_2: \mathbb{B}\times \mathbb{W}_0 \to \mathbb{W}_0$ and due to the definition of $\mathbb{W}_0$, the paths of $\Pi_2$ trivially start in zero and are continuous. The $(\bar{\mathcal{F}}^{\omega}_t)$-adaptedness of $(\Pi_2(t))_{t \geq 0}$ is obvious for every $\omega \in \Omega$, so it remains to verify that there exists $N_0 \in \mathcal{F}$ with $\mathbb{P}(N_0)=0$ such that for $\omega \in N_0^c$ we have \begin{equation}\label{indepone}
	\Pi_2(t)-\Pi_2(s) \text{ is } \mathbb{P}_{\omega}\text{ -independent of } \bar{\mathcal{F}}^{\omega}_s \text{ and } \mathbb{P}\circ(\bar{W}_t-\bar{W}_s)^{-1} = \mathbb{P}_{\omega}\circ (\Pi_2(t)-\Pi_2(s))^{-1}
	\end{equation}for all $s, t \in \mathbb{Q}$ with $0 \leq s < t$, because then the assertion follows by an approximation of arbitrary $s,t \in \mathbb{R}$ through suitable $s_n, t_n \in \mathbb{Q}$.
	To prove (\ref{indepone}) we fix $0 \leq s < t$ with $s,t \in \mathbb{Q}$, choose $A_1 \in \mathcal{B}_s(\mathbb{B}), A_2 \in \mathcal{B}_s(\mathbb{W}_0), A_0 \in \mathcal{F}$ arbitrary and obtain for $y \in \bar{U}$:
	\begin{align*}
	&\int_{A_0}\mathbb{E}_{\mathbb{P}_{\omega}}\big[\text{exp}(i\langle y,\Pi_2(t)-\Pi_2(s)\rangle_{\bar{U}})\mathds{1}_{A_1\times A_2}\big]\mathbb{P}(\text{d}\omega) = \int_{A_0}\mathbb{E}_{\mathbb{P}}\big[\text{exp}(i\langle y,\bar{W}_t-\bar{W}_s\rangle _{\bar{U}})\mathds{1}_{A_1}(X)\mathds{1}_{A_2}(\bar{W})|\mathcal{F}_0\big]\mathbb{P}(\text{d}\omega)\\&=\int_{A_0}\mathbb{E}_{\mathbb{P}}\big[\text{exp}(i\langle y,\bar{W}_t-\bar{W}_s\rangle _{\bar{U}})\big]\mathbb{E}_{\mathbb{P}}\big[\mathds{1}_{A_1}(X)\mathds{1}_{A_2}(\bar{W})|\mathcal{F}_0\big]\mathbb{P}(\text{d}\omega)\\&=\int_{A_0}\mathbb{E}_{\mathbb{P}}\big[\text{exp}(i\langle y,\bar{W}_t-\bar{W}_s\rangle _{\bar{U}})\big]\mathbb{P}_{\omega}(A_1\times A_2)\mathbb{P}(\text{d}\omega).
	\end{align*}Above we used Remark \ref{Rem rcd} (iv) for the first and last equality and the independence of $\bar{W}_t-\bar{W}_s$ and $\mathcal{F}_s$ in the second equation. By varying $A_0$ in $\mathcal{F}$, we obtain $\mathbb{P}$-a.s.:
	\begin{align}\label{al}
	\notag	\mathbb{E}_{\mathbb{P}_{\omega}}\big[\text{exp}(i\langle y,\Pi_2(t)-\Pi_2(s)\rangle_{\bar{U}})\mathds{1}_{A_1\times A_2}\big] &= \mathbb{E}_{\mathbb{P}}\big[\text{exp}(i\langle y,\bar{W}_t-\bar{W}_s\rangle _{\bar{U}})\big]\mathbb{P}_{\omega}(A_1\times A_2) \\&= \mathbb{E}_{\mathbb{P}_{\omega}}[\text{exp}(i\langle y,\Pi_2(t)-\Pi_2(s)\rangle _{\bar{U}})]\mathbb{P}_{\omega}(A_1\times A_2).\end{align}The last equality follows by the independence of $\bar{W}_t-\bar{W}_s$ from $\mathcal{F}_0$ and again Remark \ref{Rem rcd} (iv). In particular, choosing $A_1 = \mathbb{B}$ and $A_2 = \mathbb{W}_0$, we obtain for all $y$ in a countable, dense subset of $\bar{U}$:
	$$\mathbb{E}_{\mathbb{P}_{\omega}}\big[\text{exp}(i\langle y,\Pi_2(t)-\Pi_2(s)\rangle_{\bar{U}})\big] = \mathbb{E}_{\mathbb{P}}\big[\text{exp}(i\langle y,\bar{W}_t-\bar{W}_s\rangle _{\bar{U}})\big]\text{ for }\mathbb{P}\text{ -a.a. }\omega \in \Omega,$$ which by the uniqueness of the Fourier-transform implies that for $\mathbb{P}$-a.a. $\omega \in \Omega $ $$\mathbb{P}_{\omega}\circ (\Pi_2(t)-\Pi_2(s))^{-1} = \mathbb{P}\circ (\bar{W}_t-\bar{W}_s)^{-1} \text{ for all }s,t \in \mathbb{Q} \text{ as above. }$$ Further note that the exception set in (\ref{al}) can, for fixed $0 \leq s < t$, be chosen independently of $A_1, A_2$, because both $\mathcal{B}_s(\mathbb{B})$ and $\mathcal{B}_s(\mathbb{W}_0)$ are countably generated. Then the usual monotone class argument, together with Lemma \ref{Independ. of right-c. filtr}, shows that $\Pi_2(t)-\Pi_2(s)$ is $\mathbb{P}_{\omega}$-independent of $\bar{\mathcal{F}}^{\omega}_s$ for $\mathbb{P}$-a.a. $\omega \in \Omega$ for all $s,t$ as above.
\end{proof}
The following statement will be crucial for the proof of Theorem \ref{UL impl JUL}. A similar result for the finite-dimensional setting is a main tool for Cherniy's result in \cite{ChernyiPaper} (c.f. Lemma 3.3. therein). Due to its importance for our main proof below, we decided to give a detailed proof of this lemma for our infinite-dimensional framework. Below $\hat{\Pi}_2$ denotes the formal standard $\mathbb{R}^{\infty}$-Wiener process associated to $\Pi_2$.
\begin{lem}\label{Long rcd lemma}
	$(\Pi_1,\hat{\Pi}_2)$ is a weak solution to Eq. (\ref{Equation}) on $(\bar{\Omega},\bar{\mathcal{F}}^{\omega},(\bar{\mathcal{F}}_t^{\omega})_{t \geq 0},\mathbb{P}_{\omega})$ with $\mathbb{P}_{\omega}\circ \Pi_1(0)^{-1}= \delta_x= \mathbb{P}\circ X_0^{-1}$ for $\mathbb{P}$-a.a. $\omega \in \Omega$.
\end{lem}
The proof is split into two steps. We work with a sequence of elementary processes $(p_n)_{n \in \mathbb{N}}$, which approximates $\sigma$ in $L^2(\text{L}_2(U,H);\mathbb{P}^X)$, because only for elementary integrands we have a pathwise definition of the stochastic integral. This pathwise definition is necessary in order to allow us to ``put $\omega$ in the integrand as well as in the integrator" and thereby ``put $\Pi_1$ and $\Pi_2$ in the right places". This step becomes apparent in (\ref{goodeq}) and in the definition of the set $\bar{B}_t$.
\begin{proof}
	Of course $\Pi_1$ is $\mathbb{B}$-valued and $(\bar{\mathcal{F}}_t^{\omega})$-adapted for every $\omega \in \Omega$. By the previous lemma, $\Pi_2$ is an $(\bar{\mathcal{F}}_t^{\omega})$-$\bar{Q}$-Wiener process on $(\bar{\Omega},\bar{\mathcal{F}}^{\omega},(\bar{\mathcal{F}}_t^{\omega})_{t \geq 0},\mathbb{P}_{\omega})$ for $\mathbb{P}$-a.a. $\omega \in \Omega$. Let $\hat{\Pi}_2$ be the associated standard $\mathbb{R}^{\infty}$-Wiener process. Concerning integrability, fix $t \geq 0$ and note that
	$$\bar{A}_t := \bigg\{(y,w) \in \mathbb{B}\times \mathbb{W}_0\,\bigg|\int_{0}^{t}||b(s,y)||_E\,\text{d}s + \int_{0}^{t}||\sigma(s,y)||^2_{\text{L}_2(U,H)}\text{d}s < +\infty\bigg\}$$ is contained in $\mathcal{B}(\mathbb{B})\otimes \mathcal{B}(\mathbb{W}_0)$ and that $1=\mathbb{P}\big((X,\bar{W}) \in \bar{A}_t\big)=\int_{\Omega}\mathbb{P}_{\omega}(\bar{A}_t)\,\mathbb{P(\text{d}\omega)}$ holds, because $(X,W)$ is a weak solution to $(\ref{Equation})$. Consequently $\mathbb{P}_{\omega} (\bar{A}_t)=1$ for $\mathbb{P}$-a.a. $\omega \in \Omega$, which implies \begin{equation*}
	\int_{0}^{t}||b(s,\Pi_1)||_E\,\text{d}s + \int_{0}^{t}||\sigma(s,\Pi_1)||^2_{\text{L}_2(U,H)}\text{d}s < +\infty \,\,\mathbb{P}_{\omega}\text{-a.s.}
	\end{equation*} for $\mathbb{P}$-a.a. $\omega \in \Omega$ and all zero sets can obviously be chosen independently of $t$.  Hence we only need to verify the following: For $T > 0$ there exists a $\mathbb{P}$-zero-set $N_2 \in \mathcal{F}$ such that for all $\omega \in N_2^c$: 
	\begin{enumerate}
		\item [(I)] $\Pi_1(t)=\Pi_1(0)+\int_{0}^{t}b(s,\Pi_1)\text{d}s+\int_{0}^{t}\sigma(s,\Pi_1)\text{d}\hat{\Pi}_2(s)\,\,\,\,\mathbb{P}_{\omega}$-a.s. for all $t \in [0,T]$ on $E$;
		\item[(II)] $\mathbb{P}_{\omega}\circ \Pi_1(0)^{-1}=\delta_x.$
	\end{enumerate}	
	We prove assertion (I) in two steps.
	\begin{enumerate}
		\item [(i)]Here we assume \begin{equation}\label{First assump}
		\mathbb{E}_{\mathbb{P}^X}\bigg[\int_{0}^{T}||\sigma(s,\cdot)||^2_{\text{L}_2(U,H)}\text{d}s\bigg]<+\infty \text{ for all }T \geq 0,
		\end{equation}where $\mathbb{P}^X$ denotes the distribution of $X:\Omega \to \mathbb{B}$. Now fix $T > 0$. Since $\text{L}_2(U,H)$ is a separable Hilbert space and $\sigma:\mathbb{R}_+\times \mathbb{B} \to \text{L}_2(U,H)$ is measurable and $(\mathcal{B}^+_t(\mathbb{B}))$-adapted, by \cite[Lemma 2.5.]{RPaper} we obtain the existence of a sequence $(p_n)_{n \in \mathbb{N}}$ of $\text{L}_2(U,H)$-valued $(\mathcal{B}_t^+(\mathbb{B}))$-predictable, elementary processes on $[0,T]\times \mathbb{B}$ such that 
		$$\mathbb{E}_{\mathbb{P}^X}\bigg[\int_{0}^{T}||\sigma(s,y)-p_n(s,y)||^2_{\text{L}_2(U,H)}\text{d}s \bigg]\underset{n \to \infty}{\to} 0,$$i.e. in particular each $p_n$ is of the form
		$p_n(s,y) = \sum_{m=0}^{j_n-1}\Phi_m^n(y)\mathds{1}_{]t^n_m,t^n_{m+1}]}(s), \,(s,y) \in [0,T]\times \mathbb{B},$ where $\Phi_m^n: \mathbb{B} \to \text{L}_2(U,H)$ is strongly $\mathcal{B}_{t_m}^+(\mathbb{B})$-measurable for every $m \in \{0,...,j_n-1\}$, has finite image and $0 = t^n_0 < ...<t^n_{j_n} = T$ is a finite partition of $[0,T]$. We immediately observe
		\begin{equation}\label{onetwo}
		\mathbb{E}\bigg[\int_{0}^{T}||\sigma(s,X)-p_n(s,X)||^2_{\text{L}_2(U,H)}\text{d}s \bigg]\underset{n \to \infty}{\to} 0
		\end{equation} and that $p_n(\cdot,X)$ is still elementary and $(\mathcal{F}_t)$-predictable. Thus $p_n(\cdot,X)\circ J^{-1} \in \Lambda^2_T(\bar{W},\bar{U},H,\mathcal{P}_T)$ and by the isometry stated in Proposition \ref{Isometry}, (\ref{onetwo}) yields 
		$$\int_{0}^{\cdot}p_n(s,X)\text{d}W_s \underset{n \to \infty}{\to}\int_{0}^{\cdot}\sigma(s,X)\text{d}W_s \text{ in } \mathcal{M}_c^2(T;H),$$ thus in particular \begin{equation}\label{align1}
		\mathbb{E}\bigg[\underset{t\in [0,T]}{\text{sup}}\big|\big|\int_{0}^{t}p_n(s,X)\text{d}W_s-\int_{0}^{t}\sigma(s,X)\text{d}W_s\big|\big|^2_H\bigg]\underset{n \to \infty}{\to} 0.
		\end{equation}
		Since conditional expectation is an $L^p$-contraction for $p \geq 1$, we obtain
		\begin{align*}
		&\bigg|\bigg|\mathbb{E}\bigg[\int_{0}^{T}||\sigma(s,X)-p_n(s,X)||^2_{\text{L}_2(U,H)}\text{d}s\big{|}\mathcal{F}_0\bigg]\bigg|\bigg|_{L^1(\Omega)}\leq \mathbb{E}\bigg[\int_{0}^{T}||\sigma(s,X)-p_n(s,X)||^2_{\text{L}_2(U,H)}\text{d}s\bigg].
		\end{align*}Hence, by (\ref{onetwo}), there exists a subsequence $(n_k)_{k \in \mathbb{N}}$ such that
		$$\mathbb{E}\bigg[\int_{0}^{T}||\sigma(s,X)-p_{n_k}(s,X)||^2_{\text{L}_2(U,H)}\text{d}s\big{|}\mathcal{F}_0\bigg] \underset{k \to \infty}{\to} 0\,\, \mathbb{P}\text{-a.s.}$$ and thereby even 
		$\mathbb{E}\bigg[\int_{0}^{t}||\sigma(s,X)-p_{n_k}(s,X)||^2_{\text{L}_2(U,H)}\text{d}s\big{|}\mathcal{F}_0\bigg] \underset{k \to \infty}{\to} 0\,\, \mathbb{P}\text{-a.s. for all }t \in [0,T].$ As a consequence, by Remark \ref{Rem rcd} (iv), we obtain that for every $t \in [0,T]$ we have
		\begin{equation}\label{Equatio}
		\mathbb{E}_{\mathbb{P}_{\omega}}\bigg[\int_{0}^{t}||\sigma(s,\Pi_1)-p_{n_k}(s,\Pi_1)||^2_{\text{L}_2(U,H)}\text{d}s\bigg] \underset{k \to \infty}{\to}0 \,\,\, \mathbb{P}\text{-a.s.}
		\end{equation}Applying the isometry for stochastic integrals once more (this time for the Wiener process $\hat{\Pi}_2$ and the admissible integrands $p_{n_k}(\cdot,\Pi_1)$ and $\sigma(s,\Pi_1)$) we conclude by (\ref{Equatio}): For every $t \in [0,T]$, for  $\mathbb{P}$-a.a. $\omega \in \Omega$ we have
		\begin{equation}\label{onethree}
		\int_{0}^{t}p_{n_k}(s,\Pi_1)\text{d}\hat{\Pi}_2(s) \underset{k \to \infty}{\to}\int_{0}^{t}\sigma(s,\Pi_1)\text{d}\hat{\Pi}_2(s) \text{ in }L^2(\bar{\Omega},\mathbb{P}_{\omega};H).
		\end{equation}\\
		Now we consider (\ref{align1}) only along the same subsequence $(n_k)_{k \in \mathbb{N}}$. Then there is a further subsequence $(n_{k_l})_{l \in \mathbb{N}}$, for which for every $t \in [0,T]$
		\begin{equation}\label{oneone}
		\int_{0}^{t}p_{n_{k_l}}(s,X)\text{d}W_s \underset{l \to \infty}{\to} \int_{0}^{t}\sigma(s,X)\text{d}W_s = X_t-X_0-\int_{0}^{t}b(s,X)\text{d}s
		\end{equation} $\mathbb{P}$-a.s. Note that since $(p_n)_{n \in \mathbb{N}}$ is a sequence of elementary processes, the stochastic integral on the left hand side in (\ref{oneone}) is defined pathwise, i.e.
		\begin{align*}
		\bigg(\int_{0}^{t}p_{n_{k_l}}(s,X)\text{d}W_s\bigg)(\omega) = \sum_{m=0}^{j_{n_{k_l}}-1}\Phi_m^{n_{k_l}}(X(\omega)) J^{-1}\big(\bar{W}_{t^{n_{k_l}}_{m+1}\wedge t}(w)-\bar{W}_{t^{n_{k_l}}_{m}\wedge t}(w)\big).
		\end{align*}For $t \in [0,T]$ the set $$\bar{B}_t := \bigg\{(y,w) \in \mathbb{B}\times \mathbb{W}_0|\sum_{m=0}^{j_{n_{k_l}}-1}\Phi_m^{n_{k_l}}(y)\circ J^{-1}(w_{t^{n_{k_l}}_{m+1}\wedge t}-w_{t^{n_{k_l}}_{m}\wedge t}) \underset{l \to \infty}{\to}y_t-y_0-\int_{0}^{t}b(s,y)\text{d}s\bigg\},$$ is obviously contained in  $\mathcal{B}(\mathbb{B})\otimes \mathcal{B}(\mathbb{W}_0)$ and (\ref{onethree}) implies  $\mathbb{P}\big(\{(X,\bar{W}) \in \bar{B}_t^c \}\big)=0$. For every $t \in [0,T]$, we conclude
		$0=\mathbb{P}\big((X,\bar{W})\in \bar{B}_t^c\big)=\int_{\Omega}\mathbb{P}_{\omega}(\bar{B}_t^c)\,\mathbb{P}(\text{d}\omega),$ which gives $\mathbb{P}_{\omega}(\bar{B}_t^c)=0$ $\mathbb{P}$-a.s. and thus in turn for $\mathbb{P}$-a.a. $\omega \in \Omega$:
		\begin{equation}\label{goodeq}
		\int_{0}^{t}p_{n_{k_l}}(s,\Pi_1)\text{d}\hat{\Pi}_2 \underset{l \to \infty}{\to} \Pi_1(t)-\Pi_1(0)-\int_0^{t}b(s,\Pi_1)\text{d}s \,\,\mathbb{P}_{\omega}\text{-a.s.}
		\end{equation}
		But now (\ref{onethree}) especially holds along the same subsequence $(n_{k_l})_{l \in \mathbb{N}}$. Choosing a further subsequence (possibly depending on $\omega$ and $t$) for which the convergence in (\ref{onethree}) holds $\mathbb{P}_{\omega}$-a.s., we conclude together with (\ref{goodeq}): For every $t \in [0,T]$ there is $N_t \in \mathcal{F}$ with $\mathbb{P}(N_t)=0$ such that for all $\omega \in N_t^c$
		\begin{equation*}
		\Pi_1(t)=\Pi_1(0)+\int_{0}^{t}b(s,\Pi_1)\text{d}s+\int_{0}^{t}\sigma(s,\Pi_1)\text{d}\hat{\Pi}_2(s) \,\,\,\,\mathbb{P}_{\omega}\text{-a.s.}
		\end{equation*} By the continuity in $E$ of all terms, the zero set $N_t$ can be chosen independently of $t \in [0,T]$. Hence this case is settled.
		\item[(ii)] In the second step we only assume 
		\begin{equation}\label{Second assump}
		\int_{0}^{T}||\sigma(s,y)||^2_{\text{L}_2(U,H)}\text{d}s < +\infty \,\,\, \mathbb{P}^X\text{-a.s.} \text{ for all }T \geq 0,
		\end{equation}
		which is automatically true, since $\int_{0}^{T}||\sigma(s,X)||^2_{\text{L}_2(U,H)}\text{d}s < +\infty \,\,\, \mathbb{P}\text{-a.s.}$ by assumption for all $T \geq 0$. Fix $T > 0$. We work with the following maps for $k \in \mathbb{N}$.
		$$\tau^T_k:\mathbb{B} \to \mathbb{R}_+,\, \tau^T_k(y) := \text{inf}\big\{s \geq 0\big|\int_{0}^{s}||\sigma(r,y)||^2_{\text{L}_2(U,H)}\text{d}r > k \big\} \wedge T, $$which, by Fubini's theorem, is an $(\mathcal{B}_t^+(\mathbb{B}))$-stopping time for every $k \in \mathbb{N}$. We continue with the following observations.
		\begin{enumerate}
			\item[(a)] For every $k \in \mathbb{N}$ and $T >0$, (\ref{First assump}) is fulfilled when one replaces $\sigma$ by $\mathds{1}_{]0,\tau^T_k]}\sigma$ and  $\mathds{1}_{]0,\tau^T_k]}\sigma: \mathbb{R}_+ \times \mathbb{B} \to \text{L}_2(U,H)$ is measurable and  $(\mathcal{B}_t^+(\mathbb{B}))$-adapted. 
			\item[(b)] Due to the continuity of $t \mapsto \int_{0}^{t}||\sigma(s,y)||^2_{\text{L}_2(U,H)}\text{d}s$ and (\ref{Second assump}), we have $\tau^T_k(X) \nearrow T$ $\mathbb{P}-$a.s. for $k \to \infty$ and hence, since $\big{\{}y \in \mathbb{B}|\tau^T_k(y) \underset{k \to \infty}{\to} T\big{\}} \in \mathcal{B}(\mathbb{B})\otimes \mathcal{B}(\mathbb{W}_0):$
			$$\mathbb{P}(\{\tau^T_k(X)) \underset{k \to \infty}{\to} T\}) = \int_{\Omega}\mathbb{P}_{\omega}(\{(y,w) \in \mathbb{B}\times \mathbb{W}_0|\tau^T_k(y) \underset{k \to \infty}{\to} T\})\,\text{d}\mathbb{P}(\omega),$$which yields $\tau^T_k(\Pi_1) \underset{k \to \infty}{\to} T$ $\mathbb{P}_{\omega}$-a.s. for $\mathbb{P}$-a.a. $\omega \in \Omega$.
		\end{enumerate}Hence, as in the previous step, we find elementary, $(\mathcal{B}_t^+(\mathbb{B}))$-predictable functions $(q^{T,k}_n)_{n \in \mathbb{N}}$ with \begin{equation*}
		\mathbb{E}\bigg[\int_{0}^{T}||\mathds{1}_{]0,\tau^T_k(X)]}\sigma(s,X)-q^{T,k}_n(s,X)||^2_{\text{L}_2(U,H)}\text{d}s \bigg]\underset{n \to \infty}{\to} 0
		\end{equation*} and therefore, by the isometry for stochastic integrals, also 
		\begin{equation}\label{argh}
		\mathbb{E}\bigg[\underset{t\in [0,T]}{\text{sup}}\big|\big|\int_{0}^{t}q^{T,k}_n(s,X)\text{d}W_s-\int_{0}^{t}\mathds{1}_{]0,\tau^T_k(X)]}\sigma(s,X)\text{d}W_s\big|\big|^2_H\bigg]\underset{n \to \infty}{\to} 0.
		\end{equation} As in (\ref{Equatio}), we find a subsequence $(n_m)_{m \in \mathbb{N}}$ such that
		\begin{equation*}
		\mathbb{E}_{\mathbb{P}_{\omega}}\bigg[\int_{0}^{T}||\mathds{1}_{]0,\tau^T_k(\Pi_1)]}\sigma(s,\Pi_1)-q^{T,k}_{n_m}(s,\Pi_1)||^2_{\text{L}_2(U,H)}\text{d}s\bigg] \underset{m \to \infty}{\to}0 \,\,\, \mathbb{P} \text{-a.s.}
		\end{equation*}Similarly to (\ref{onethree}) we obtain for $\mathbb{P}$-a.a. $\omega \in \Omega:$
		\begin{equation}\label{onesix}\begin{aligned}
			\int_{0}^{T}q^{T,k}_{n_m}(s,\Pi_1)\text{d}\hat{\Pi}_2(s) \underset{m \to \infty}{\to}\int_{0}^{T}\mathds{1}_{]0,\tau^T_k(\Pi_1)]}\sigma(s,\Pi_1)\text{d}\hat{\Pi}_2(s)= \int_{0}^{\tau^T_k(\Pi_1)}\sigma(s,\Pi_1)\text{d}\hat{\Pi}_2(s) 
		\end{aligned}
		\end{equation}
		in $L^2(\bar{\Omega},\mathbb{P}_{\omega};H)$. Considering (\ref{argh}) along the same subsequence $(n_m)_{m \in\mathbb{N}}$ yields a further subsequence $(n_{m_l})_{l \in \mathbb{N}}$ with
		\begin{align*}
		\int_{0}^{T}q^{T,k}_{n_{m_l}}(s,X)\text{d}W_s \underset{l \to \infty}{\to} \int_{0}^{\tau^T_k(X)}\sigma(s,X)\text{d}W_s = X_{\tau^T_k(X)}-X_0-\int_{0}^{\tau^T_k(X)}b(s,X)\text{d}s \,\,\, \mathbb{P}\text{-a.s.}
		\end{align*} Proceeding along the same steps as in part (i) up to (\ref{goodeq}) with the necessary technical adjustments, we arrive at
		\begin{equation*}
		\int_{0}^{T}q^{T,k}_{n_{m_l}}(s,\Pi_1)\text{d}\hat{\Pi}_2 \underset{l \to \infty}{\to} \Pi_1(\tau^T_k(\Pi_1))-\Pi_1(0)-\int_0^{\tau^T_k(\Pi_1)}b(s,\Pi_1)\text{d}s \,\,\mathbb{P}_{\omega}\text{-a.s.}
		\end{equation*} for $\mathbb{P}$-a.a. $\omega \in \Omega$. Comparing with (\ref{onesix}), we observe $\mathbb{P}_{\omega}$-a.s.
		\begin{equation}\label{einszwei}
		\Pi_1(\tau^T_k(\Pi_1))=\Pi_1(0)+\int_{0}^{\tau^T_k(\Pi_1)}b(s,\Pi_1)\text{d}s+\int_{0}^{\tau^T_k(\Pi_1)}\sigma(s,\Pi_1)\text{d}\hat{\Pi}_2(s)
		\end{equation}for $\mathbb{P}$-a.a. $\omega \in \Omega$. Now consider (\ref{einszwei}) for all $k \in \mathbb{N}$ simultaneously and pass to the limit of $\tau^T_k(\Pi_1)$ for $k \to \infty$, which, as we stated above, is $\mathbb{P}_{\omega}$-a.s. equal to $T$ for $\mathbb{P}$-a.a. $\omega \in \Omega$. By the continuity of all terms involved, for $\mathbb{P}$-a.a. $\omega \in \Omega$
		\begin{equation*}
		\Pi_1(T)=\Pi_1(0)+\int_{0}^{T}b(s,\Pi_1)\text{d}s+\int_{0}^{T}\sigma(s,\Pi_1)\text{d}\hat{\Pi}_2(s) \,\,\,\,\mathbb{P}_{\omega}\text{-a.s.}
		\end{equation*}Repeating this procedure for every $T >0$ and using the continuity of both sides of the equation as $E$-valued processes, we obtain the statement.
	\end{enumerate}
	Finally consider (II). Due to $X_0 \equiv x$, we have for each $A \in \mathcal{B}(H)$:
	\begin{align}\label{oneeleven}
	\{0,1\} \ni \mathbb{P}(X_0 \in A) = \mathbb{P}\big((X,\bar{W})\in \{\Pi_1(0) \in A\}\big) = \int_{\Omega}\mathbb{P}_{\omega}\big(\{\Pi_1(0)\in A\}\big)\mathbb{P}(\text{d}\omega)
	\end{align}and thereby 	$\mathbb{P}_{\omega}\big(\{\Pi_1(0)\in A\}\big) = \mathbb{P}(X_0 \in A)$ for $\mathbb{P}$-a.a. $\omega \in \Omega.$
	Since $H$ is a separable Hilbert space, we can choose a $\cap$-stable, countable generator of $\mathcal{B}(H)$. Then the above equality holds for all elements $A$ of this generating set outside one common $\mathbb{P}$-zero set and from there we conclude $\mathbb{P}_{\omega}\circ \Pi_1(0)^{-1} = \mathbb{P}\circ X_0^{-1}$ for $\mathbb{P}$-a.a. $\omega \in \Omega$ as measures on $\mathcal{B}(H)$, which finishes the proof.  \end{proof}Throughout the proof of our main results we will work with stochastic integrals, which involve certain projection-valued operators as integrands. The next lemma states that these integrals are well-defined.
\begin{lem}\label{Stoch Int well-defined for psi,phi}
	Let $(X,W)$ be a weak solution to Eq. (\ref{Equation}) on a stochastic basis $(\Omega, \mathcal{F}, (\mathcal{F}_t)_{t \geq 0},\mathbb{P})$. For $(t,y)\in \mathbb{R}_+\times \mathbb{B}$ define the operators $\phi(t,y),\psi(t,y) \in \text{L}(U)$ through $$\phi(t,y)(u) := \text{pr}_{\text{ker}\,\sigma(t,y)^{\perp}}(u) \text{ and } \psi(t,y)(u):=\text{pr}_{\text{ker}\,\sigma(t,y)}(u),$$where $\text{pr}_{V}(\cdot)$ denotes the orthogonal projection onto a closed linear subspace $V \subseteq U$. Then the following holds:
	\begin{enumerate}
		\item [(i)] As processes in $(t,y) \in \mathbb{R}_+ \times \mathbb{B}$, $\phi$ and $\psi$ are $\text{L}_2(U,H)$-valued, measurable and $(\mathcal{B}_t(\mathbb{B}))$-adapted with respect to the strong Borel $\sigma$-algebra on $\text{L}_2(U,H)$.
		\item[(ii)] For any $\mathbb{R}^{\infty}$-Wiener process $W'$ on $(\Omega, \mathcal{F}, (\mathcal{F}_t)_{t \geq 0},\mathbb{P})$, the stochastic integrals $\int_{0}^{t}J \circ\phi(s,X)\text{d}W'_s$ and $\int_{0}^{t}J \circ \psi(s,X)\text{d}W'_s$ are well-defined, $\bar{U}$-valued continuous processes for $t \geq 0$. Further for every $T >0$, both processes are square-integrable on $[0,T]$ in the sense that $J \circ \phi(\cdot,X) \circ J^{-1}, J \circ \psi(\cdot,X)\circ J^{-1} \in \Lambda^2_T(\bar{W'},\bar{U},\bar{U},\mathcal{P}_T)$ for every $T >0$.
	\end{enumerate}
\end{lem}
\begin{proof}
	\begin{enumerate}
		\item [(i)] Due to the obvious identity 
		$\phi(t,y) = \text{id}_{U} - \psi(t,y),$ it suffices to prove the assertion for $(t,y) \mapsto \psi(t,y)$. Hence we fix $u \in U$ and must prove that $\psi(u) : \mathbb{R}_+ \times \mathbb{B} \to H$, $\psi(u)(t,y) := \psi(t,y)(u)$ is measurable and $(\mathcal{B}_t(\mathbb{B}))$-adapted. But this can be done as in \cite[Lemma 9.2]{Ondrejat}.
		\item[(ii)] By (i) and because $J \in \text{L}(U,\bar{U})$, both $J\circ \phi(\cdot,X)$ and $J \circ \psi(\cdot,X)$ are strongly measurable, $(\mathcal{B}_t(\mathbb{B}))$-adapted and $\text{L}_2(U,\bar{U})$-valued.
		Now fix $(t,y) \in \mathbb{R}_+\times \mathbb{B}$. For $A \in \text{L}_2(U,\bar{U})$ the value  $||A||_{\text{L}_2(U,\bar{U})} = \big(\sum_{k=1}^{\infty}||Af_k||^2_{\bar{U}}\big)^{\frac{1}{2}}$ is independent of the orthonormal basis $\{f_k\}_{k \in \mathbb{N}}$. Hence we may choose $\{f_k\}_{k \in \mathbb{N}}$ such that either $f_k \in \text{ker}\,\sigma(t,y)$ or $f_k \in \text{ker}\,\sigma(t,y)^{\perp}$ for every $k \in \mathbb{N}$. Then we obtain
		\begin{align*}
		||J \phi(t,y) ||_{\text{L}_2(U,\bar{U})}^2 = \sum_{f_k \in \, \text{ker}\,\sigma(t,y)^{\perp}}||Jf_k||^2_{\bar{U}} \leq  ||J||_{\text{L}_2(U,\bar{U})}^2 < +\infty
		\end{align*}for all $(t,y) \in \mathbb{R}_+ \times \mathbb{B}$ since $J$ is Hilbert-Schmidt. Hence for each $t \geq 0$ $$\mathbb{E}\bigg[\int_{0}^{t}||J \phi(s,X)||_{\text{L}_2(U,\bar{U})}^2\text{d}s\bigg] \leq \mathbb{E}\bigg[\int_{0}^{t}||J||^2_{\text
			L_2(U,\bar{U})}\text{d}s\bigg] = t||J||^2_{\text
			L_2(U,\bar{U})}< \infty,$$which completes the proof of (ii), because the $\psi$-integral can be treated similarly. \qedhere \end{enumerate}\end{proof}
Our next goal is to prove that the quadratic cross variation of two stochastic integrals is additive, if the integrators are independent Wiener processes (c.f. (\ref{qu var add}) below). We will need this result along the proof of our second main theorem. We start with a technical lemma. Its proof is postponed to the appendix.
\begin{lem}\label{AAA}
	Let $(\Omega, \mathcal{F}, (\mathcal{F}_t)_{t \geq 0}, \mathbb{P})$ be a stochastic basis, $Q \in \text{L}^+_1(U)$ and $W^1, W^2$ two independent $U$-valued $(\mathcal{F}_t)$-$Q$-Wiener processes on $\Omega$. Then for every $\phi_1$, $\phi_2: \mathbb{R}_+ \times \Omega \to \text{Lin}(U,\mathbb{R})$ with $\phi_k \in \Lambda^2_T(W^k,U,\mathbb{R},\mathcal{P}_T)$ for every $T >0$ and $k \in \{1,2\}$, the following holds:
	\begin{equation*}
	\mathbb{E}\bigg[\int_{0}^{\tau}\phi_1(s)\text{d}W^1(s)\cdot \int_{0}^{\tau}\phi_2(s)\text{d}W^2(s)\bigg] = 0
	\end{equation*}for every bounded $(\mathcal{F}_t)$-stopping time $\tau:\Omega \to \mathbb{R}_+$. In particular the covariation process of the two stochastic integrals above is constantly zero $\mathbb{P}$-a.s.
\end{lem}

Now we can straight forward prove the desired result:
\begin{prop}\label{good cor}
	Let $\phi_k \in \Lambda^2_T(W^k,U,H,\mathcal{P}_T)$ for every $T > 0$ for $k \in \{1,2\}$ and $W^1$, $W^2$ as above. Then we have $\mathbb{P}$-a.s.:
	\begin{equation}\label{qu var add}
	\begin{aligned}
	\ll \int_{0}^{\cdot}\phi_1(s) \text{d}W^1(s) +  \int_{0}^{\cdot}\phi_2(s)\text{d} W^2(s) \gg_t\, = \,	\ll \int_{0}^{\cdot}\phi_1(s) \text{d}W^1(s) \gg_t + \ll  \int_{0}^{\cdot}\phi_2(s)\text{d} W^2(s) \gg_t
	\end{aligned}
	\end{equation}for every $t \geq 0$.
\end{prop}
\begin{proof}Let $(f_k)_{k \in \mathbb{N}}$ be an orthonormal basis of $H$.
	Lemma \ref{AAA} and the fact that bounded linear operators interchange with stochastic integrals imply for every $i,j \in \mathbb{N}:$ $$	\langle \langle \,\, \langle \int_{0}^{\cdot}\phi_1(s) \text{d}W^1(s),f_i \rangle_H, \langle \int_{0}^{\cdot}\langle\phi_2(s)\text{d} W^2(s),f_j \rangle_H \,\, \rangle \rangle_t\, = 0  \text{ for all }t\geq 0\,\,\, \mathbb{P}\text{-a.s.},$$ because by assumption on $\phi_k$, the integrands $\langle \phi_k(\cdot),f_i\rangle_H$ obviously fulfill the assumption of the previous lemma for every $k \in \{1,2\}$ and $i \in \mathbb{N}$. Hence the assertion follows by Corollary \ref{Quad Var additive for independent}.
\end{proof}
Finally we present a definition, which will be useful within the proof of Theorem \ref{UL impl JUL}.
\begin{dfn}
	Let $H$ be a separable Hilbert space with inner product $\langle \cdot, \cdot \rangle_H$. The Hilbert space $(H\oplus H, \langle \cdot, \cdot \rangle_{H \oplus H})$ is defined as the Cartesian product $H\times H$ with the inner product $\langle (h_1,h_2),(h_3,h_4)\rangle_{H \oplus H} := \langle h_1,h_3\rangle_H +\langle h_2,h_4 \rangle_H$. When no confusion is possible, we abbreviate $\langle \cdot,\cdot \rangle_{H \oplus H}$ by $\langle \cdot, \cdot \rangle_{\oplus}$.
\end{dfn}
\begin{rem}
	It is obvious that the Hilbert space $(H\oplus H, \langle \cdot, \cdot \rangle_{H \oplus H})$ is separable and that $\mathcal{B}(H \oplus H) = \mathcal{B}(H) \otimes \mathcal{B} (H)$. The latter holds, because the metric induced by $\langle \cdot, \cdot \rangle_{\oplus}$ induces the product topology on $H \oplus H$.
\end{rem}
\subsection{Proofs of the main results}
Now we give proofs for the two main results of this paper. \\ \\
\textbf{Proof of Theorem \ref{Main Theorem I}:}
Fix a measure $\mu \in \mathcal{M}^+_1(\mathcal{B}(H))$ for which joint uniqueness in law given $\mu$ holds, a stochastic basis $(\Omega, \mathcal{F},(\mathcal{F}_t)_{t \geq 0}, \mathbb{P})$, an $(\mathcal{F}_t)$-$\mathbb{R}^{\infty}$-Wiener process  $W$ and an $\mathcal{F}_0$-measurable map $\xi_0:\Omega \to H$ with $\mathbb{P}\circ \xi_0^{-1} = \mu.$ Let $Z:=F_{\mathbb{P}\circ{\xi_0}^{-1}}(\xi_0, \bar{W})$ be a strong solution with respect to this data. We prove \begin{equation}\label{Z=X in MainProof2}
Z_t=X_t \; \text{for all} \,\, t \geq 0 \; \mathbb{P}\text{-a.s.}
\end{equation}for every weak solution $X$ with respect to the same data.
To do so, let $F$, $Z$ and $X$ be as above and set $\mathcal{G}^{\bar{W}}_0 := \overline{\sigma(\xi_0,\bar{W})}^{\,\mathbb{P}}$. As before, $\bar{W}$ denotes the $\bar{U}$-valued $(\mathcal{F}_t)$-$\bar{Q}$-Wiener process associated to $W$. We make the following observations:\begin{enumerate}
	\item [(i)] $\mathcal{E}^{\bar{W}}_0:=\{\big(\xi_0^{-1}(G) \cap \bar{W}^{-1}(B)\big) \cup N|G \in \mathcal{B}(H), B \in \mathcal{B}(\mathbb{W}_0), \,\mathbb{P}(M)=0\}$ is a $\cap$-stable generator of $\mathcal{G}_0^{\bar{W}}$.
	\item [(ii)]Since by definition of the strong solution $F$, $(h,y) \mapsto F_{\mathbb{P}\circ \xi_0^{-1}}(h,y)$ is $\overline{\mathcal{B}(H)\otimes\mathcal{B}(\mathbb{W}_0)}^{\mu \otimes \mathbb{P}^{\bar{Q}}}/\mathcal{B}(\mathbb{B})$-measurable, the map $\omega \mapsto F_{\mathbb{P}\circ \xi_0^{-1}}(\xi_0(\omega),\bar{W}(\omega))$ $=Z(w)$ is $\mathcal{G}^{\bar{W}}_0/\mathcal{B}(\mathbb{B})$-measurable. Indeed, as $\xi_0$ is $\mathcal{F}_0$-measurable and hence $\mathbb{P}$-independent of $\bar{W}$, we obtain $\mathbb{P}\circ (\xi_0,\bar{W})^{-1} = \mu \otimes \mathbb{P}^{\bar{Q}}$ and thereby the claim follows by the $\mathcal{G}^{\bar{W}}_0/\overline{\mathcal{B}(H)\otimes\mathcal{B}(\mathbb{W}_0)}^{\mu \otimes \mathbb{P}^{\bar{Q}}}$- measurability of $(\xi_0,\bar{W}): \Omega \to H\times \mathbb{W}_0$. Here  $\mathbb{P}^{\bar{Q}}$ denotes the measure $\mathbb{P}\circ \bar{W}^{-1}$ on $\mathcal{B}(\mathbb{W}_0)$.
	\end{enumerate}
	Since $(\mathbb{B},\mathcal{B}(\mathbb{B}))$ is Polish there exists a unique regular conditional distribution of $Z: \Omega \to \mathbb{B}$ with respect to $\mathcal{G}^{\bar{W}}_0,$ which we denote by $(Q^Z_{\omega})_{\omega \in \Omega}$. Since $Z$ is $\mathcal{G}^{\bar{W}}_0$-measurable, Remark \ref{Rem rcd} implies $Q^Z_{\omega}=\delta_{Z(w)}$ $\mathbb{P}$-a.s.\\
	As we assume joint uniqueness in law given $\mu$ and we have $X_0=\xi_0=Z_0$ $\mathbb{P}$-a.s. and $\mathbb{P}\circ \xi_0^{-1} = \mu$, we obtain \begin{equation}\label{joint law equal}\mathbb{P}\circ (X,\bar{W})^{-1} = \mathbb{P}\circ(Z,\bar{W})^{-1}. \end{equation}
	By the same arguments as above there exists a unique regular conditional distribution of $X: \Omega \to \mathbb{B}$ with respect to $\mathcal{G}^{\bar{W}}_0$, which we denote by $(Q^X_{\omega})_{\omega \in \Omega}$.
	Clearly $\omega \mapsto \delta_{Z(w)}(A)$ is $\mathcal{G}^{\bar{W}}_0$-measurable for every $A \in \mathcal{B}(\mathbb{B})$. Further, due to (\ref{joint law equal}), we have
	$
	\mathbb{P}(\{X \in A\}\cap \{\bar{W}\in B\}) = \mathbb{P}(\{Z \in A\}\cap \{\bar{W}\in B\})\text{ for all }A \in \mathcal{B}(\mathbb{B}), \, B \in \mathcal{B}(\mathbb{W}_0).
	$ Since $\{\pi_0 \in G\} \in \mathcal{B}(\mathbb{B})$ for $G \in \mathcal{B}(H)$ and $X_0=\xi_0=Z_0 \; \mathbb{P}\text{-a.s.}$, we obtain
	\begin{equation}\label{fast alles gleich bis auf nullmengen}\mathbb{P}(\{\xi_0 \in G\}\cap \{X \in A\}\cap \{\bar{W}\in B\}) = \mathbb{P}(\{\xi_0 \in G\}\cap \{Z \in A\}\cap \{\bar{W}\in B\})
	\end{equation} for arbitrary $G \in \mathcal{B}(H), \, A \in \mathcal{B}(\mathbb{B}), \, B \in \mathcal{B}(\mathbb{W}_0) $. For fixed $A \in \mathcal{B}(\mathbb{B})$ set $\mathbb{P}^X_A(\cdot):=\mathbb{P}(\{X \in A\}\cap \cdot)$ and $\mathbb{P}^Z_A(\cdot):=\mathbb{P}(\{Z \in A\}\cap \cdot)$ on $\mathcal{G}_0^{\bar{W}}$. Then (\ref{fast alles gleich bis auf nullmengen}) yields $\mathbb{P}^X_A(E)=\mathbb{P}^Z_A(E) \text{ for all } E \in \mathcal{E}^{\bar{W}}_0,$ whence we conclude $\mathbb{P}^X_A=\mathbb{P}^Z_A$ $\text{ for all } A \in \mathcal{B}(\mathbb{B})$ as measures on $\mathcal{G}^{\bar{W}}_0$, i.e.
$
	\mathbb{P}(\{X \in A\}\cap C) = \mathbb{P}(\{Z \in A\} \cap C) \text{ for all } C \in \mathcal{G}^{\bar{W}}_0.
$ We conclude $\label{rcd of X is deltaZ}
	Q^X_{\omega}=\delta_{Z(w)} \; \mathbb{P}\text{-a.s.}
	$
	Hence for every $A \in \mathcal{B}(\mathbb{B})$\begin{equation*}\label{equality for indikatorf}\mathbb{E}[\mathds{1}_{\{X \in A\}}|\mathcal{G}^{\bar{W}}_0](\omega)=\delta_{Z(\omega)}(A) = \mathds{1}_{\{Z \in A\}}(\omega)\,\,\,\mathbb{P}\text{-a.s.}, \end{equation*}where the exception set may depend on $A$. Thus $\label{g-equality for all nonneg}
	\mathbb{E}[g(X)|\mathcal{G}^{\bar{W}}_0]=g(Z) \; \mathbb{P}\text{-a.s.}
	$ for every bounded and  $\mathcal{B}(\mathbb{B})/\mathcal{B}(\mathbb{R})\text{-measurable} \;\,g:\mathbb{B} \to \mathbb{R}$ by a simple monotone class argument. For each such $g$ we note 
	$
	\mathbb{E}[(g(X)-g(Z))^2]=2\mathbb{E}[g(Z)^2]-2\mathbb{E}\big[\mathbb{E}[g(X)g(Z)|\mathcal{G}^{\bar{W}}_0]\big]=2\mathbb{E}[g(Z)^2]-2\mathbb{E}[g(Z)g(Z)] =0.
	$ The first equality follows by the equality in law of $X$ and $Z$. Thus we obtain \begin{equation}\label{einszweidrei}
	g(X)=g(Z) \,\,\,\mathbb{P}\text{-a.s.}
	\end{equation}
	for each bounded, measurable $g: \mathbb{B} \to \mathbb{R}$. Now we can finally verify (\ref{Z=X in MainProof2}): Fix an orthonormal basis $\{f_i\}_{i \in \mathbb{N}}$ of $H$ and set $\sigma_j:H\to \mathbb{R}$, $\sigma_j=\langle\cdot, f_j\rangle_{H}$.
	For $q \in \mathbb{Q}_+$ and $j, n \in \mathbb{N}$, define $g_q^{j,n}: \mathbb{B} \to \mathbb{R}$ through $$g_q^{j,n}(y):=\big(\sigma_j(\pi_q(y))\wedge n\big)\vee -n,\,\,\, y \in \mathbb{B}
	$$and note that these functions are clearly bounded and $\mathcal{B}(\mathbb{B})/\mathcal{B}(\mathbb{R})$-measurable. As above, $\pi_q: \mathbb{B} \to H$ denotes the canonical projection from $\mathbb{B}$ to $H$ at time $q$. We have $\underset{n\to \infty}{\lim}g_q^{j,n}(X(\omega))=\langle X_q(w),f_j\rangle_{H}$ for every $\omega \in \Omega$. Applying (\ref{einszweidrei}) to $g_q^{j,n}$ for every $q,j,n$, we obtain $X_q=Z_q$ for all $ q \in \mathbb{Q}_+$ $\mathbb{P}$-a.s. and the path-continuity of $X$ and $Z$ in $H$ completes the proof. 
	\qed
	\begin{rem}
		The theorem and its proof remain valid if one replaces the assumption on the existence of a strong solution by the following\\
		\\	$\mathbf{Generalized \,\,assumption:}$ \\For every triple 
		$\big((\Omega, \mathcal{F},(\mathcal{F}_t)_{t \geq 0}, \mathbb{P}), W, \xi_0\big)$ for which at least one weak solution $X$ exists (i.e. the pair $(X,W)$ is a weak solution on this stochastic basis with $X_0 = \xi_0$ $\mathbb{P}$-a.s.), there also exists a solution $Z: \Omega \to \mathbb{B}$ subject to this triple, which is $\mathcal{G}_0^{\bar{W}}/\mathcal{B}(\mathbb{B})$-measurable.
		\end{rem}
		We now turn to the proof of Theorem \ref{Main Theorem I}. We will heavily need several properties and computation rules of stochastic integrals with respect to arbitrary square-integrable, continuous martingales. These properties are well-known to experts on stochastic integration in infinite dimensions. Nevertheless, for the convenience of the reader, we review the construction and properties of such stochastic integrals in Appendix B. \\ \\
		\noindent\textbf{Proof of Theorem \ref{UL impl JUL}}:
	Fix $x \in H$ and assume uniqueness in law given $\delta_x$ holds. We prove the following: For any weak solution $(X,W)$ to 
		\begin{equation}\label{det Eq}
		X_t = x+\int_{0}^{t}b(s,X)\text{d}s+\int_{0}^{t}\sigma(s,X)\text{d}W_s,\,\, t \geq 0
		\end{equation}on a stochastic basis $(\Omega, \mathcal{F}, (\mathcal{F}_t)_{t \geq 0}, \mathbb{P})$, the joint distribution $\mathbb{P}\circ(X,\bar{W})^{-1}$ is uniquely determined by $\mathbb{P}\circ X^{-1}$.
		Here and for the rest of the proof, for a $\mathbb{R}^{\infty}$-Wiener process $W$ we denote by $\bar{W}$ the $\bar{U}$-valued $\bar{Q}$-Wiener process associated to $W$. As we pointed out in the recap on cylindrical Wiener processes in the second section, we have $\bar{Q}  = JJ^*$. Let us fix a weak solution $(X,W)$ to Eq. (\ref{det Eq}). \\
		\\Let $(\Omega', \mathcal{F}', (\mathcal{F}'_t)_{t \geq 0}, \mathbb{P}')$ be another stochastic basis and $\bar{W}^1$, $\bar{W}^2$ two independent   $\bar{U}$-valued $(\mathcal{F}'_t)$-$\bar{Q}$-Wiener processes on this basis, i.e. \begin{equation*}
		\bar{W}^1 = \sum_{k=1}^{\infty}\beta_k^1 J e_k,\thickspace \bar{W}^2 = \sum_{k=1}^{\infty}\beta_k^2 J e_k,
		\end{equation*}where $(\beta_k^i)_{\{{k \in \mathbb{N},\, i \in \{1,2\}}\}}$ is an independent family of $\mathbb{R}$-valued $(\mathcal{F}'_t)$-Brownian motions on $\Omega'$ and $(e_k)_{k \in \mathbb{N}}$ is the orthonormal basis of $U$ we fixed in Subsection 1.2. The $\mathbb{R}^{\infty}$-Wiener processes associated to $\bar{W}^1$ and $\bar{W}^2$, i.e. the families $(\beta_k^1)_{k \in \mathbb{N}}$ and $(\beta_k^2)_{k \in \mathbb{N}}$, will be denoted by $W^1$ and $W^2$, respectively.\\
		\\ Define $(\tilde{\Omega}, \tilde{\mathcal{F}}, (\tilde{\mathcal{F}}_t)_{t \geq 0}, \tilde{\mathbb{P}})$ := $(\Omega \times \Omega', \overline{{\mathcal{F}\otimes \mathcal{F}'}}^{\,\mathbb{P} \otimes\mathbb{P}'}, (\tilde{\mathcal{F}_t})_{t \geq 0}, \mathbb{P}\otimes \mathbb{P}')$, (where $\tilde{\mathcal{F}_t}:= \bigcap_{\epsilon > 0}\sigma(\mathcal{F}_{t+\epsilon} \otimes \mathcal{F}'_{t + \epsilon}, \tilde{\mathcal{N}})),$ which is a stochastic basis. Here we set $\tilde{\mathcal{N}}$ := $\{A \in \tilde{\mathcal{F}}|\tilde{\mathbb{P}}(A)=0\}$. Define the processes $\tilde{X}, \tilde{\bar{W}}, \tilde{\bar{W}}^1, \tilde{\bar{W}}^2$ on $(\tilde{\Omega}, \tilde{\mathcal{F}}, (\tilde{\mathcal{F}}_t)_{t \geq 0}, \tilde{\mathbb{P}})$ through
		$\tilde{X}((\omega_1,\omega_2)) := X(\omega_1),
		\tilde{\bar{W}}((\omega_1,\omega_2)):= \bar{W}(\omega_1),
		\tilde{\bar{W}}^i((\omega_1,\omega_2)):=\bar{W}^i(\omega_2), \,i \in \{1,2\},$ for $(\omega_1,\omega_2) \in \tilde{\Omega}$ and analogously for the $\mathbb{R}^{\infty}$-Wiener processes $W, W^1$ and $W^2$. Clearly $\tilde{\bar{W}}, \tilde{\bar{W}}^i$ are independent $\bar{U}$-valued $(\tilde{\mathcal{F}_t})$-$\bar{Q}$-Wiener processes on $\tilde{\Omega}$ and 
		$\tilde{W}, \tilde{W}^1, \tilde{W}^2$ are independent $\mathbb{R}^{\infty}$-Wiener processes.
		Note that we also have \begin{equation}\label{zwoa}
		\tilde{\bar{W}} = \bar{\tilde{W}} = \sum_{k=1}^{\infty}\tilde{\beta}_kJe_k.
		\end{equation} We obtain that $(\tilde{X}, \tilde{W})$ is a weak solution to Eq. (\ref{det Eq}) on $(\tilde{\Omega}, \tilde{\mathcal{F}}, (\tilde{\mathcal{F}}_t)_{t \geq 0}, \tilde{\mathbb{P}})$ with $\tilde{X}_0 \equiv x \,\,\,\, \tilde{\mathbb{P}}\text{-a.s.},$ because the $(\tilde{\mathcal{F}}_t)$-adaptedness of $(\tilde{X}_t)_{t \geq 0}$ is trivial and all properties, which hold $\mathbb{P}$-a.s. for X also hold $\tilde{\mathbb{P}}$-a.s. for $\tilde{X}$.\\
		\\ For $t \geq 0$ and $y \in \mathbb{B}$, let $\phi(t,y) \in \text{L}(U)$ be the orthogonal projection onto ker$\,\sigma(t,y)^{\bot} \subseteq U$ and $\psi(t,y) \in \text{L}(U)$ the orthogonal projection onto ker$\,\sigma(t,y)$. By Lemma \ref{Stoch Int well-defined for psi,phi} the stochastic integrals of $J \circ \phi(s,\tilde{X})$ and $J \circ \psi(s,\tilde{X})$ with respect to $\tilde{W},\tilde{W}^1,\tilde{W}^2$ are well-defined. For $t \geq 0$ we define the processes
		$$
		\bar{V}_t^1:= \int_{0}^{t}J \phi(s,\tilde{X})\text{d}\tilde{W}_s + \int_{0}^{t}J  \psi(s,\tilde{X})\text{d}\tilde{W}^1_s,\,\,\,
		\bar{V}_t^2 :=\int_{0}^{t}J  \phi(s,\tilde{X})\text{d}\tilde{W}^2_s + \int_{0}^{t}J  \psi(s,\tilde{X})\text{d}\tilde{W}_s,$$
		which are clearly continuous, $\bar{U}$-valued local $(\tilde{\mathcal{F}}_t)$-martingales on $(\tilde{\Omega},\tilde{\mathcal{F}},(\tilde{\mathcal{F}}_t)_{t \geq 0}, \tilde{\mathbb{P}})$. We collect the following properties of $\phi$ and $\psi$: For each $(t,y) \in \mathbb{R}_+ \times \mathbb{B}$ we have
		\begin{align}
		\phi(t,y) = \phi(t,y)^* &\text{ and } \psi(t,y) = \psi(t,y)^* \label{phi adjoint},\\
		\phi(t,y)^2 = \phi(t,y) &\text{ and } \psi(t,y)^2=\psi(t,y) 	\label{phi proj},\\
		\phi(t,y)+&\psi(t,y) = \text{id}_{U} \label{phi psi id},\\
		\sigma(t,y) = \sigma(t,y) \circ \phi(t,y&) \text{ and } \sigma(t,y) \circ \psi(t,y) = 0_{U} 	\label{ sigma 0 },\\
		\phi(t,y)\circ \psi(t,y)& = 0_{U} = \psi(t,y)\circ \phi(t,y) \label{composition 0}.
		\end{align}
	
		We will now verify that $\bar{V}^1$ and $\bar{V}^2$ are $\tilde{\mathbb{P}}$-independent $\bar{U}$-valued $(\tilde{\mathcal{F}}_t)$-$\bar{Q}$-Wiener processes on $\tilde{\Omega}$.
		\begin{enumerate}
			\item $\bar{V}^1$, $\bar{V}^2$ are $(\tilde{\mathcal{F}}_t)$-$\bar{Q}$-Wiener processes on $(\tilde{\Omega},\tilde{\mathcal{F}},(\tilde{\mathcal{F}}_t)_{t \geq 0}, \tilde{\mathbb{P}})$:\\
			For every $T >0$, both processes are clearly square-integrable, continuous martingales on $[0,T]$ and thus $\bar{V}^1$, $\bar{V}^2 \in \mathcal{M}_c^2(T;\bar{U})$. Hence by the Lévy-characterization (c.f. Proposition \ref{Generalized Levy}), applied to arbitrarily large $T > 0$, it suffices to prove $\ll\bar{V}^i\gg_t = t\bar{Q}\, \,\tilde{\mathbb{P}}\text{-a.s.} \text{ for all } t \geq 0 $ for $i \in \{1,2\}.$ We calculate:
			\begin{align*}
			\ll\bar{V}^1\gg_t &= \int_{0}^{t}(J \phi(s,\tilde{X}) J^{-1} \bar{Q}^{\frac{1}{2}})(J  \phi(s,\tilde{X}) J^{-1} \bar{Q}^{\frac{1}{2}})^*\text{d}s\int_{0}^{t}(J  \psi(s,\tilde{X}) J^{-1} \bar{Q}^{\frac{1}{2}})(J  \psi(s,\tilde{X}) J^{-1} \bar{Q}^{\frac{1}{2}})^*\text{d}s\\&=\int_{0}^{t}J \phi(s,\tilde{X}) J^{-1}  \bar{Q}  (J \phi(s,\tilde{X})J^{-1})^* \text{d}s +\int_{0}^{t}J \psi(s,\tilde{X}) J^{-1}  \bar{Q}  (J \psi(s,\tilde{X}) J^{-1})^* \text{d}s\\&=\int_{0}^{t}J \phi(s,\tilde{X})\phi(s,\tilde{X})^*  J^*+J  \psi(s,\tilde{X}) \psi(s,\tilde{X})^*  J^*\text{d}s\\&=\int_{0}^{t}J  \big(\phi(s,\tilde{X})+\psi(s,\tilde{X})\big) J^*\text{d}s = t\bar{Q},\,\,\,t \geq 0\,\, \tilde{\mathbb{P}}\text{-a.s.}
			\end{align*}In the above calculation we used Proposition \ref{good cor} together with Proposition \ref{Shape qu var, Q, a for stoch int} in the first, $\bar{Q}=JJ^*$ and elementary computation rules for adjoint operators in the second and third, (\ref{phi adjoint}) and (\ref{phi proj}) in the fourth and (\ref{phi psi id}) in the fifth equation. Likewise we obtain 
			$\ll\bar{V}^2\gg_t = t\bar{Q}, \,\,\, t \geq 0\,\, \tilde{\mathbb{P}}\text{-a.s.}$ and therefore $\bar{V}^1$ and $\bar{V}^2$ are $(\tilde{\mathcal{F}}_t)$-$\bar{Q}$-Wiener processes.\\
			\item  $\bar{V}^1$ and $\bar{V}^2$ are $\tilde{\mathbb{P}}$-independent: \\Define $\bar{Q}^{\oplus} \in \text{L}_1^+(\bar{U}\oplus \bar{U})$ through $\bar{Q}^{\oplus}\big((\bar{u}_1,\bar{u}_2)\big) := (\bar{Q}\bar{u}_1,\bar{Q}\bar{u}_2)$. Note that $(\bar{V}^1,\bar{V}^2)$ is clearly a continuous local $\bar{U}\oplus \bar{U}$-valued $(\tilde{\mathcal{F}}_t)$-martingale. We want to prove
		$
			\ll (\bar{V}^1,\bar{V}^2)\gg_t = t\bar{Q}^{\oplus}\,\,\tilde{\mathbb{P}}\text{-a.s.} \text{ for all }t \in [0,T]
			$ for every $T \geq 0$. By Proposition \ref{Prop Qu Var} this is equivalent to 
			\begin{align}\label{kajau}
			& \notag \langle (\bar{V}^1_t,\bar{V}^2_t), (a_1,b_1) \rangle _{\oplus} \cdot \langle(\bar{V}^1_t,\bar{V}^2_t), (a_2,b_2) \rangle _{\oplus} - \langle t\bar{Q}^{\oplus}(a_1,b_1),(a_2,b_2)\rangle _{\oplus}
			\end{align}being an $(\tilde{\mathcal{F}}_t)$-martingale for all $a_1,a_2,b_1,b_2 \in \bar{U}$ and on every $[0,T]$. By definition of $\langle \cdot, \cdot \rangle_{\oplus}$ and since both 
			$$\langle \bar{V}^1_t,a_1\rangle_{\bar{U}} \cdot \langle \bar{V}^1_t,a_2\rangle_{\bar{U}} - \langle t\bar{Q}a_1,a_2\rangle_{\bar{U}} \text{ and } \langle \bar{V}^2_t,b_1\rangle_{\bar{U}} \cdot \langle \bar{V}^2_t,b_2\rangle_{\bar{U}} - \langle t\bar{Q}b_1,b_2\rangle_{\bar{U}}$$ are martingales for all $a_1,a_2,b_1,b_2 \in \bar{U}$, this holds if and only if
			$$\langle \bar{V}^1_t,a_1 \rangle _{\bar{U}} \cdot \langle \bar{V}^2_t,b_2 \rangle _{\bar{U}} + \langle \bar{V}^2_t,b_1 \rangle _{\bar{U}} \cdot \langle \bar{V}^1_t,a_2 \rangle _{\bar{U}}$$ is an $(\mathcal{\tilde{F}}_t)$-martingale for all $a_1,a_2,b_1,b_2 \in \bar{U}$ on $[0,T]$ for all $T >0$. Hence fix $T >0$, $a,b \in \bar{U}$ and consider $\big(\langle \bar{V}^1_t,a \rangle _{\bar{U}} \cdot \langle \bar{V}^2_t,b \rangle _{\bar{U}}\big)_{t \in [0,T]}.$ After multiplying out and interchanging the linear functionals $\langle \cdot, a\rangle _{\bar{U}}$ and $\langle \cdot , b \rangle _{\bar{U}}$ with the stochastic integrals, it is clear by definition of $\bar{V}^1$ and $\bar{V}^2$ and due to Lemma \ref{AAA} that every summand but \begin{equation}\label{one}
			\int_{0}^{t}\langle J \phi(s,\tilde{X}) J^{-1}(\cdot), a \rangle _{\bar{U}}\,\text{d}\bar{\tilde{W}}_s \cdot \int_{0}^{t}\langle J\psi(s,\tilde{X}) J^{-1}(\cdot), b \rangle _{\bar{U}}\,\text{d}\bar{\tilde{W}}_s
			\end{equation} is an $(\mathcal{\tilde{F}}_t)$-martingale on $[0,T]$. Using Lemma 2.4.5 in \cite{RSPDE} we further express the stochastic integrals in (\ref{one}) through
			\begin{align}\label{converg}
			\int_{0}^{t}\langle J\circ \phi(s,\tilde{X})\circ J^{-1}(\cdot), a \rangle _{\bar{U}}\,\text{d}\bar{\tilde{W}}_s  = \sum_{k=1}^{\infty}\int_{0}^{t}\langle J \phi(s,\tilde{X}) J^{-1}(Je_k),a\rangle _{\bar{U}} \,\text{d}\tilde{\beta}_k(s), 
			\end{align}$t \in [0,T] \,\,\mathbb{\tilde{P}}\text{-a.s.}$, where the limit is taken in $L^2\big(\tilde{\Omega}, \mathcal{\tilde{F}},\tilde{\mathbb{P}};C([0,T],\mathbb{R})\big)$ and analogously for the second integral. Here $\tilde{\beta}_k$ and $e_k$ are as in (\ref{zwoa}). We calculate as follows.
			\begin{align}\label{oneseven}
			&\notag \langle \langle \sum_{k=1}^{\infty}\int_{0}^{\cdot}\langle J \phi(s,\tilde{X}) J^{-1}(Je_k),a\rangle _{\bar{U}} \,\text{d}\tilde{\beta}_k(s), \sum_{l=1}^{\infty}\int_{0}^{\cdot}\langle J \psi(s,\tilde{X}) J^{-1}(Je_l),b\rangle _{\bar{U}} \,\text{d}\tilde{\beta}_l(s) \rangle \rangle_t \\&=\notag \bigg(\sum_{k=1}^{\infty}\sum_{l=1}^{\infty}\langle \langle \int_{0}^{\cdot}\langle J \phi(s,\tilde{X}) J^{-1}(Je_k),a\rangle _{\bar{U}} \,\text{d}\tilde{\beta}_k(s), \int_{0}^{\cdot}\langle J \psi(s,\tilde{X}) J^{-1}(Je_l),b\rangle _{\bar{U}} \,\text{d}\tilde{\beta}_l(s)\rangle \rangle\bigg)_t\\& =\bigg(\sum_{k=1}^{\infty}\int_{0}^{\cdot}\langle J \phi(s,\tilde{X}) e_k,a\rangle _{\bar{U}}\cdot\langle J \psi(s,\tilde{X}) e_k,b\rangle _{\bar{U}}\,\text{d}s\bigg)_t \\& =  \notag \int_{0}^{t}\sum_{k=1}^{\infty}\langle J \phi(s,\tilde{X}) e_k,a\rangle _{\bar{U}}\cdot\langle J \psi(s,\tilde{X}) e_k,b\rangle _{\bar{U}}\,\text{d}s =\notag \int_{0}^{t}\sum_{k=1}^{\infty}\langle e_k,\phi(s,\tilde{X})J^*(a)\rangle _U \cdot \langle e_k, \psi(s,\tilde{X})J^*(b)\rangle _U\,\text{d}s \\&=\notag \int_{0}^{t}\langle \phi(s,\tilde{X})J^*(a),\psi(s,\tilde{X})J^*(b)\rangle_U \, \text{d}s =\notag\int_{0}^{t}\langle J^*(a),\phi(s,\tilde{X})\psi(s,\tilde{X})J^*(b)\rangle_U \, \text{d}s = 0  \,\,\, \mathbb{\tilde{P}}\text{-a.s.}
			\end{align}The first equality is due to the convergence on the right-hand side in (\ref{converg}) in $L^2\big(\Omega,\mathcal{F},\mathbb{P}; C([0,T],\mathbb{R})\big)$ and due to the uniqueness of  the covariation process of continuous martingales. For the third equality, consider (\ref{oneseven}) along a subsequence $(N_l)_{l \in \mathbb{N}}$ for which 
			$$\sum_{k=1}^{N_l}\int_{0}^{\cdot}\langle J \phi(s,\tilde{X}) e_k,a\rangle _{\bar{U}}\cdot\langle J \psi(s,\tilde{X}) e_k,b\rangle _{\bar{U}}\,\text{d}s$$ converges uniformly to $ \bigg(\sum_{k=1}^{\infty}\int_{0}^{\cdot}\langle J \phi(s,\tilde{X}) e_k,a\rangle _{\bar{U}}\cdot\langle J \psi(s,\tilde{X}) e_k,b\rangle _{\bar{U}}\,\text{d}s\bigg)_{t\in [0,T]}$ on $[0,T]$ $\mathbb{P}$-a.s. for $l \to +\infty$. Then we clearly have for all $t \in [0,T]$
			\begin{align*}&\bigg(\sum_{k=1}^{\infty}\int_{0}^{\cdot}\langle J \phi(s,\tilde{X}) e_k,a\rangle _{\bar{U}}\cdot\langle J \psi(s,\tilde{X}) e_k,b\rangle _{\bar{U}}\,\text{d}s\bigg)_t = \underset{l \to \infty}{\text{lim}}\sum_{k=1}^{N_l}\int_{0}^{t}\langle J \phi(s,\tilde{X}) e_k,a\rangle _{\bar{U}}\cdot\langle J \psi(s,\tilde{X}) e_k,b\rangle _{\bar{U}}\,\text{d}s\\&= \int_{0}^{t}\underset{l\to \infty}{\text{lim}}\sum_{k=1}^{N_l}\langle J \phi(s,\tilde{X}) e_k,a\rangle _{\bar{U}}\cdot\langle J \psi(s,\tilde{X}) e_k,b\rangle _{\bar{U}}\text{d}s= \int_{0}^{t}\sum_{k=1}^{\infty}\langle J \phi(s,\tilde{X})e_k,a \rangle_{\bar{U}} \cdot \langle J \psi(s,\tilde{X})e_k,b \rangle_{\bar{U}} \text{d}s\end{align*}$\mathbb{P}$-a.s., since we can interchange the limit with the integral, because for fixed $\omega \in \Omega$ the function $t \mapsto ||\phi(t,\tilde{X}(\omega))J^*a||_U \cdot ||\psi(t,\tilde{X}(\omega))J^*b||_U$ is, by Cauchy-Schwarz-inequality, a dominating $L^1([0,T],\text{d}t;\mathbb{R})$-function of the sequence 
			$$\bigg(\sum_{k=1}^{N_l}\langle J \phi(s,\tilde{X}(\omega)) e_k,a\rangle _{\bar{U}}\cdot\langle J \psi(s,\tilde{X}(\omega)) e_k,b\rangle _{\bar{U}}\bigg)_{l \in \mathbb{N}}, $$so that Lebesgue's dominated convergence theorem applies. The last expression equals zero because of (\ref{composition 0}). Hence $(\bar{V}^1,\bar{V}^2)$ is an $(\tilde{\mathcal{F}}_t)$-$\bar{Q}^{\oplus}$-Wiener process.
			Consequently we have the following expression $\mathbb{\tilde{P}}$-a.s. independently of $t \geq 0$:
			\begin{equation}\label{oneeight}
			(\bar{V}_t^1,\bar{V}_t^2) = \sum_{i=1}^{\infty}\sqrt{\bar{\lambda}_{i}}\beta'_{i}(t)\bar{f}_{i},
			\end{equation}where $\bar{f}_i$ is defined through $\bar{f}_i := (f_{\frac{i+1}{2}},0)$ for $i \in 2\mathbb{N}_0+1$ and $\bar{f}_i := (0,f_{\frac{i}{2}})$ for $i \in 2\mathbb{N}$ and the series converges in $L^2(\Omega, \mathcal{F},\mathbb{P};C([0,T],\bar{U}\oplus \bar{U}))$ for every $T >0$. Here $\{f_n|n \in \mathbb{N}\}$ denotes an orthonormal basis of $\bar{U}$ consisting of eigenvectors of $\bar{Q}$. It is obvious that $\{\bar{f}_n|n \in \mathbb{N}\}$ is an orthonormal basis of $\bar{U} \oplus \bar{U}$ consisting of eigenvectors of $\bar{Q}^{\oplus}$. Further $\bar{\lambda}_n$ is the corresponding eigenvalue of $\bar{f}_n$ and $\{\beta'_n|n \in \mathbb{N}\}$ is an independent family of real-valued $(\tilde{\mathcal{F}}_t)$-Brownian motions on $\tilde{\Omega}$. From the definition of $\bar{f}_n$ and (\ref{oneeight}) we immediately obtain $\mathbb{P}$-a.s.:
			\begin{equation*}
			\bar{V}^1_t = \sum_{i \in \mathbb{N}}\sqrt{\bar{\lambda}_{2i-1}}\beta'_{2i-1}(t)f_i \text{ and } 	\bar{V}^2_t = \sum_{i \in \mathbb{N}}\sqrt{\bar{\lambda}_{2i}}\beta'_{2i}(t)f_i,\,\,\,t \geq 0.
			\end{equation*}Since the $\sigma$-algebras $\sigma(\beta'_n(t)|t \geq 0,n \in 2\mathbb{N}_0+1)$ and $\sigma(\beta'_n(t)|t\geq 0,n \in 2\mathbb{N})$ are $\mathbb{\tilde{P}}$-independent and clearly
			$
				\sigma(\bar{V}^1_t|t\geq 0) \subseteq \sigma(\beta'_n(t)|t \geq 0,n \in 2\mathbb{N}_0+1),\,\,\, 	\sigma(\bar{V}^2_t|t\geq 0) \subseteq \sigma(\beta'_n(t)|t \geq 0,n \in 2\mathbb{N}),
		$ we have proved the independence of $(\bar{V}^1_t)_{t \geq 0}$ and $(\bar{V}^2_t)_{t \geq 0}$.
			\end{enumerate}
			In the sequel we will use the notation $V^i$ for the formal  $\mathbb{R}^{\infty}$-Wiener process associated to $\bar{V}^i$. The next step is to prove that $(\tilde{X},V^1)$ is a weak solution to (\ref{det Eq}) on $(\tilde{\Omega},\tilde{\mathcal{F}},(\tilde{\mathcal{F}}_t)_{t \geq 0}, \tilde{\mathbb{P}})$ (in fact even with respect to the bigger filtration $(\mathcal{G}^+_t)_{t \geq 0}$ as we shall see below) and that $\tilde{X}$ and $\bar{V}^2$ are $\tilde{\mathbb{P}}$-independent. 
			\begin{enumerate}
				\item $(\tilde{X},V^1)$ is a weak solution to (\ref{det Eq}) on $(\tilde{\Omega},\tilde{\mathcal{F}},(\tilde{\mathcal{F}}_t)_{t \geq 0}, \tilde{\mathbb{P}})$: \\We prove
				$\int_{0}^{t}\sigma(s,\tilde{X})\text{d}\tilde{W}_s = \int_{0}^{t}\sigma(s,\tilde{X})\text{d}V_s^1,\,\,\,t\geq 0\,\, \tilde{\mathbb{P}}\text{-a.s.}$ Applying (\ref{ sigma 0 }), (\ref{composition 0}) and for the second equality Proposition \ref{localize} (i), we get
				\begin{align}
				&\int_{0}^{t}\sigma(s,\tilde{X})\text{d}\tilde{W}_s \notag= \int_{0}^{t}\sigma(s,\tilde{X}) J^{-1} J \phi(s,\tilde{X}) J^{-1}\text{d}\tilde{\bar{W}}_s\notag=\int_{0}^{t}\sigma(s,\tilde{X}) J^{-1}\text{d}\bigg(\int_{0}^{s}J\phi(r,\tilde{X}) J^{-1}\text{d}\tilde{\bar{W}}_r\bigg) \notag \\&
				=\int_{0}^{t}\sigma(s,\tilde{X}) J^{-1}\text{d}\bigg(\int_{0}^{s}J  \phi(r,\tilde{X}) J^{-1}\text{d}\tilde{\bar{W}}_r+\int_{0}^{s}J \phi(r,\tilde{X}) \psi(r,\tilde{X}) J^{-1}\text{d}\tilde{\bar{W}}^1_r\bigg)\notag\\&
				=\int_{0}^{t}\sigma(s,\tilde{X}) J^{-1}\text{d}\bigg(\int_{0}^{s}J\phi(r,\tilde{X}) J^{-1} J \phi(r,\tilde{X}) J^{-1}\text{d}\tilde{\bar{W}}_r\notag+\int_{0}^{s}J  \phi(r,\tilde{X}) J^{-1} J \psi(r,\tilde{X}) J^{-1}\text{d}\tilde{\bar{W}}^1_r\bigg)\notag.
				\end{align}Note that we can indeed apply Proposition \ref{localize} due to Lemma \ref{Stoch Int well-defined for psi,phi}. Applying Proposition \ref{Asso stoch int}, we can further rewrite the integrator of the last term in the upper chain of equations as follows:
				\begin{align*}
				&\int_{0}^{s}J \phi(r,\tilde{X}) J^{-1}J  \phi(r,\tilde{X}) J^{-1}\text{d}\tilde{\bar{W}}_r+\int_{0}^{s}J\phi(r,\tilde{X}) J^{-1} J \psi(r,\tilde{X}) J^{-1}\text{d}\tilde{\bar{W}}^1_r\notag\\&\label{Int2}=\int_{0}^{s}J\phi(r,\tilde{X})J^{-1}\text{d}\bigg(\int_{0}^{r}J \phi(\alpha,\tilde{X}) J^{-1}\text{d}\tilde{\bar{W}}_{\alpha}\bigg)+\int_{0}^{s}J\phi(r,\tilde{X})J^{-1}\text{d}\bigg(\int_{0}^{r}J \psi(\alpha,\tilde{X}) J^{-1}\text{d}\tilde{\bar{W}}^1_{\alpha}\bigg)\\&
				=\int_{0}^{s}J\phi(r,\tilde{X}) J^{-1}\text{d}\bar{V}^1_r.\notag
				\end{align*} Finally, let us again apply Proposition \ref{localize} (i) and the two chains of equations above to obtain the following:
				\begin{align*}
				&\int_{0}^{t}\sigma(s,\tilde{X})\text{d}\tilde{W}_s=\int_{0}^{t}\sigma(s,\tilde{X}) J^{-1}\text{d}\bigg(\int_{0}^{s}J\phi(r,\tilde{X}) J^{-1}\text{d}\bar{V}^1_r\bigg)\\&=\int_{0}^{t}\sigma(s,\tilde{X}) \phi(s,\tilde{X})J^{-1}\text{d}\bar{V}^1_s = \int_{0}^{t}\sigma(s,\tilde{X})\text{d}V^1_s,
				\end{align*}which holds $\tilde{\mathbb{P}}$-a.s. for each $t \geq 0$ with zero set independent of $t \geq 0$.\\
				\item $\tilde{X}$ and $\bar{V}^2$ are independent on $\tilde{\Omega}$ with respect to $\tilde{\mathbb{P}}:$\\
				We first show that $(\tilde{X},\bar{V}^1)$ remains a weak solution when replacing the filtration $(\tilde{\mathcal{F}}_t)_{t \geq 0}$ by $(\mathcal{G}^+_t)_{t \geq 0}$, which is the right-continuous filtration associated to
				$\mathcal{G}_t := \tilde{\mathcal{F}}_t \vee \sigma(\bar{V}^2_s|s \geq 0), \, t \geq 0.$ By Lemma \ref{Independ. of right-c. filtr} we only need to show that $\bar{V}^1$ is a $(\mathcal{G}_t)$-$\bar{Q}$-Wiener process. Obviously $$\mathcal{G}_t = \tilde{\mathcal{F}}_t \vee \sigma(\bar{V}^2_s|s > t) = \tilde{\mathcal{F}}_t \vee \sigma(\bar{V}^2_s-\bar{V}^2_t|s \geq t).$$
				Since $(\bar{V}^1,\bar{V}^2)$ is an $(\tilde{\mathcal{F}}_t)_{t \geq 0}$-$\bar{Q}^{\oplus}$-Wiener process on $(\tilde{\Omega},\tilde{\mathcal{F}},(\tilde{\mathcal{F}}_t)_{t \geq 0}, \tilde{\mathbb{P}})$, Lemma \ref{TopLemma} implies the independence of $	\tilde{\mathcal{F}}_t$ and $\sigma(\bar{V}^1_s-\bar{V}^1_t|s \geq t) \vee \sigma(\bar{V}^2_s-\bar{V}^2_t|s \geq t) \text{ for }t \geq 0$.
				Therefore for $A_t \in \tilde{\mathcal{F}}_t, D \in \sigma(\bar{V}^2_s-\bar{V}^2_t|s \geq t), B \in \sigma(\bar{V}^1_s-\bar{V}^1_t|s \geq t)$ we have
				\begin{equation*}
				\tilde{\mathbb{P}}(B \cap D \cap A_t) = \tilde{\mathbb{P}}(B \cap D) \cdot \tilde{\mathbb{P}}(A_t)=\tilde{\mathbb{P}}(B)\cdot \tilde{\mathbb{P}}(D) \cdot \tilde{\mathbb{P}}(A_t) = \tilde{\mathbb{P}}(B)\cdot \tilde{\mathbb{P}}(D \cap A_t).
				\end{equation*}Since sets of the form $A_t \cap D$ for $A_t$ and $D$ as above form a $\cap$-stable generator of $\mathcal{G}_t$, we obtain the independence of $\sigma(\bar{V}^1_s-\bar{V}^1_t|s \geq t) $ and  $\mathcal{G}_t,$
				which yields that $\bar{V}^1$ is a $(\mathcal{G}_t)$-$\bar{Q}$-Wiener process on $\tilde{\Omega}$.
				\\ \\To obtain the desired independence of $\tilde{X}$ and $\bar{V}^2$, we now apply Lemma \ref{Long rcd lemma} to the weak solution $(\tilde{X},\bar{V}^1)$ on $(\tilde{\Omega},\tilde{\mathcal{F}},(\mathcal{G}^+_t)_{t \geq 0}, \tilde{\mathbb{P}})$ and obtain that for $\tilde{\mathbb{P}}$-a.a. $\tilde{\omega} \in \tilde{\Omega}$ the pair $(\Pi_1, \hat{\Pi}_2)$ is a weak solution on $(\bar{\Omega},\bar{\mathcal{F}}^{\tilde{\omega}},(\bar{\mathcal{G}}_t^{+\tilde{\omega}})_{t \geq 0},\mathbb{P}_{\tilde{\omega}})$ with $\mathbb{P}_{\tilde{\omega}}\circ \Pi_1(0)^{-1} = \delta_x = \tilde{\mathbb{P}}\circ \tilde{X}^{-1}_0.$ Here $(\mathbb{P}_{\tilde{\omega}})_{\tilde{\omega}\in \tilde{\Omega}}$ is the regular conditional distribution of $(\tilde{X},\bar{V}^1): \Omega \to \mathbb{B}\times \mathbb{W}_0$ with respect to $\mathcal{G}^+_0$. All other notations are as in Lemma \ref{Long rcd lemma}. By assumption, uniqueness in law given $\delta_x$ holds for the stochastic equation. Hence the measures $\mathbb{P}_{\tilde{\omega}}\circ \Pi_1^{-1}$ on $\mathcal{B}(\mathbb{B})$ are the same for $\tilde{\mathbb{P}}$-a.a. $\tilde{\omega} \in \tilde{\Omega}.$ Therefore we have for all $D \in \mathcal{G}_0^+$ and $A \in \mathcal{B}(\mathbb{B})$:
				\begin{align*}
				&\tilde{\mathbb{P}}(D \cap \{\tilde{X} \in A\}) = \int_{D}\mathds{1}_{\{(\tilde{X},\bar{V}^1) \,\in \,\Pi_1^{-1}(A)\}}\text{d}\tilde{\mathbb{P}}(\tilde{\omega})= \int_{D}\mathbb{P}_{\tilde{\omega}}(\Pi_1^{-1}(A))\text{d}\tilde{\mathbb{P}}(\tilde{\omega}) = \mathbb{P}_{\tilde{\omega}}(\Pi_1^{-1}(A))\cdot \tilde{\mathbb{P}}(D)\\&= \int_{\tilde{\Omega}}\mathbb{P}_{\tilde{\omega}}(\Pi_1^{-1}(A))\text{d}\tilde{\mathbb{P}}(\tilde{\omega})\cdot\tilde{\mathbb{P}}(D) = \tilde{\mathbb{P}}\big((\tilde{X},\bar{V}^1)^{-1}(\Pi_1^{-1}(A))\big)\cdot \tilde{\mathbb{P}}(D) = \tilde{\mathbb{P}}(\tilde{X} \in A)\cdot \tilde{\mathbb{P}}(D).
				\end{align*}The third equality holds because the map $\tilde{\omega }\mapsto \mathbb{P}_{\tilde{\omega}}(\Pi_1^{-1}(A))$ is $\tilde{\mathbb{P}}$-a.s. constant for every $A \in \mathcal{B}(\mathbb{B})$. But this shows that $\tilde{X}$ and $\mathcal{G}^+_0$ are $\tilde{\mathbb{P}}$-independent. By definition of the filtration $(\mathcal{G}^+_t)_{t \geq 0}$, then also $\tilde{X}$ and $\bar{V}^2$ are $\tilde{\mathbb{P}}$-independent.
				\end{enumerate}
				For the final step of the proof define $\chi: \mathbb{R}_+\times \mathbb{B} \to \text{Lin}(H,U)$ with domain $\mathcal{D}(\chi(t,y)):= \text{Im}\,\sigma(t,y)$ for every $(t,y) \in [0,T]\times \mathbb{B}$ through  $\chi(t,y) := \sigma(t,y)^{-1}_{\text{ker}^{\perp}}$. Here $\sigma(t,y)^{-1}_{\text{ker}^{\perp}}$ denotes the inverse of $\sigma(t,y)$ from $\text{Im}\,\sigma(t,y)$ to $\text{ker}\,\sigma(t,y)^{\perp}$. We note $\chi(t,y)\circ \sigma(t,y) = \phi(t,y)$ for all $(t,y) \in \mathbb{R}_+\times \mathbb{B}$. Using Proposition \ref{localize} (ii) for the second equality below, we obtain
				\begin{equation}\label{difficult integral}
				\int_{0}^{t}J \phi(s,\tilde{X})\text{d}\tilde{W}_s= \int_{0}^{t}J \chi(s,\tilde{X}) \sigma(s,\tilde{X})\text{d}\tilde{W}_s = \int_{0}^{t}J \chi(s,\tilde{X})\text{d}N_s,\,\,\,t \geq 0
				\end{equation}$\tilde{\mathbb{P}}$-a.s., where $\label{N}
				N_t := \tilde{X}_t-\tilde{X}_0-\int_{0}^{t}b(s,\tilde{X})\text{d}s.
				$ We continue with
				\begin{align}\label{needthis}
				\bar{\tilde{W}}_t \notag&= \int_{0}^{t}J  \text{id}_{U} J^{-1}\text{d}\bar{\tilde{W}}_s=\int_{0}^{t}J (\phi(s,\tilde{X})+\psi(s,\tilde{X})) J^{-1}\text{d}\bar{\tilde{W}}_s \notag\\&= \int_{0}^{t}J \chi(s,\tilde{X})\text{d}N_s+\int_{0}^{t}J \psi(s,\tilde{X}) J^{-1} J  \phi(
				s,\tilde{X})\text{d}\tilde{W}^2_s+\int_{0}^{t}J\psi(s,\tilde{X}) J^{-1}  J  \psi(s,\tilde{X})\text{d}\tilde{W}_s\notag\\&=\int_{0}^{t}J\chi(s,\tilde{X})\text{d}N_s+\int_{0}^{t}J\psi(s,\tilde{X}) J^{-1}\text{d}\bar{V}_s^2,\,\,\, t\geq 0 \,\, \tilde{\mathbb{P}}\text{-a.s.}
				\end{align}For the third equality, (\ref{difficult integral}) and the identities (\ref{composition 0}) and (\ref{phi proj}) are applied. The last one holds due to the linearity in the integrator and Proposition \ref{Asso stoch int}. As the first summand of (\ref{needthis}) is a measurable functional of $\tilde{X}$ and $\bar{V}^2$ is independent of $\tilde{X}$, we conclude that $\tilde{\mathbb{P}}\circ(\tilde{X},\bar{\tilde{W}})$ is uniquely determined by $\tilde{\mathbb{P}}\circ\tilde{X}$. We elaborate this step in more detail at the end of Appendix A. Since $\tilde{X} = X \circ \pi_1$ and $\tilde{\bar{W}} =  \bar{W} \circ \pi_1$ we obtain
				$$\mathbb{P}(X \in A, \bar{W} \in B) = \tilde{\mathbb{P}}(\tilde{X} \in A, \tilde{\bar{W}}\in B) \text{ for all } A \in \mathcal{B}(\mathbb{B}), B \in \mathcal{B}(\mathbb{W}_0),$$ where $\tilde{\pi}_1:\tilde{\Omega} \to \Omega$,  $\tilde{\pi}_1((\omega,\omega')) := \omega$ for $(\omega,\omega') \in \tilde{\Omega}$. Therefore also $\mathbb{P}\circ(X,\bar{W})^{-1}$ is uniquely determined by $\mathbb{P}\circ X^{-1}$, because clearly $\tilde{\mathbb{P}}\circ \tilde{X}^{-1} = \mathbb{P}\circ X^{-1}$. Hence we have proved joint uniqueness in law given $\delta_x$. The ``in particular"-assertion of the statement is now a trivial consequence of what we just proved. \qed 
\\ \\
\noindent{\large\textbf{Appendix}}
\appendix
\renewcommand{\thesection}{A}
\section{Auxiliary lemmata and proofs}
Again, let $H$ be a separable, (infinite-dimensional) real Hilbert space.
\begin{lem}\label{Independ. of right-c. filtr}
	Let $M$ be a continuous $H$-valued stochastic process on a probability space $(\Omega, \mathcal{F}, \mathbb{P})$, which is adapted to a not necessarily right-continuous filtration $(\mathcal{F}_t)_{t \geq 0}$. If $M_t-M_s$ is independent of $\mathcal{F}_s$ for all $0 \leq s < t$, then $M_t-M_s$ is also independent of $\mathcal{F}^+_s$ for all $0\leq s < t$, where $(\mathcal{F}^+_t)_{t \geq 0}$ denotes the right-continuous filtration associated to $(\mathcal{F}_t)_{t \geq 0}$.
\end{lem}
\begin{proof}
	It suffices to prove the following claim: For $0 \leq s < t$, $X:= M_t-M_s$, $\mathcal{O} \subseteq H$ open and $A_s \in \mathcal{F}_s^+$ we have
	$
	\mathbb{E}[\mathds{1}_{X \in \mathcal{O}}\cdot \mathds{1}_{A_s}] = \mathbb{E}[\mathds{1}_{X \in \mathcal{O}}]\cdot \mathbb{E}[\mathds{1}_{A_s}].
	$ For such $\mathcal{O} \subseteq H$ there exist continuous functions $(
	f_n)_{n \in \mathbb{N}}$ such that $f_n(H) \subseteq [0,1]$ for every $n \in \mathbb{N}$ and $f_n \nearrow \mathds{1}_{\mathcal{O}}$ pointwise. Hence, by Lebesgue's dominated convergence theorem it suffices to verify 
	\begin{equation}\label{0}
	\mathbb{E}[f (X) \cdot \mathds{1}_{A_s}] = \mathbb{E}[f ( X)]\cdot \mathbb{E}[\mathds{1}_{A_s}]
	\end{equation}for every continuous $f: H \to [0,1]$. By assumption, $f \circ (M_t-M_{s+\frac{1}{n}})$ is independent of $\mathcal{F}_{s+\frac{1}{n}}$ for any $n \in \mathbb{N}$. Hence we get \begin{equation}\label{1}
	\mathbb{E}[f \circ (M_t-M_{s+\frac{1}{n}})\cdot \mathds{1}_{A_s}] = \mathbb{E}[f \circ (M_t-M_{s+\frac{1}{n}})]\cdot \mathbb{E}[\mathds{1}_{A_s}]
	\end{equation} and the continuity of $M$ implies $f \circ (M_t-M_{s+\frac{1}{n}}) \underset{n \to \infty}{\to} f \circ X$ $\mathbb{P}$-a.s. Hence, and since every $f \circ (M_t-M_{s+\frac{1}{n}})$ is bounded by 1 allows to apply Lebesgue on both sides of (\ref{1}). Hence taking limits on both sides in this equation, we obtain (\ref{0}), which proves the assertion.
\end{proof}
\begin{lem}\label{TopLemma}
	Let $Q \in \text{L}^+_1(H)$ and $ Q' \in \text{L}^+_1(H\oplus H)$. Let $W^1$ and $W^2$ be two $H$-valued $(\mathcal{F}_t)$-$Q$-Wiener processes on a stochastic basis $(\Omega, \mathcal{F}, (\mathcal{F}_t)_{t \geq 0}, \mathbb{P})$ such that $(W^1,W^2)$ is an $H\oplus H$-valued $(\mathcal{F}_t)$-$Q'$-Wiener process on $(\Omega, \mathcal{F}, (\mathcal{F}_t)_{t \geq 0}, \mathbb{P})$. Then
	$\mathcal{F}_s$ is independent of $\sigma(W^1_t-W^1_s|t \geq s) \vee \sigma(W^2_t-W^2_s|t \geq s)$ for all $s \geq 0$.
\end{lem}
\begin{proof}
	Since by assumption $(W^1,W^2)$ is a Wiener process with respect to $(\mathcal{F}_t)_{t \geq 0}$, the independence of $\mathcal{F}_s$ and $\sigma\big((W^1_t,W^2_t)-(W^1_s,W^2_s)|t \geq s\big)$ for all $s \geq 0$ follows. Hence it suffices to show
	\begin{equation*}\label{gh}
	\sigma\big((W^1_t,W^2_t)-(W^1_s,W^2_s)|t \geq s\big) \supseteq \sigma(W^1_t-W^1_s|t \geq s) \vee \sigma(W^2_t-W^2_s|t \geq s).
	\end{equation*}Indeed, for $t \geq s$: $
	(W^1_t-W^1_s)^{-1}(A) = \big((W^1_t,W^2_t)-(W^1_s,W^2_s)\big)^{-1}(A\times H) \in \sigma\big((W^1_t,W^2_t)-(W^1_s,W^2_s)|t \geq s\big)$ for all $A \in \mathcal{B}(H)$. Proceeding in the same way for $W^2$, we obtain the assertion.
\end{proof}
\noindent \textbf{Proof of Lemma \ref{AAA}}: Fix $\phi_1$ and $\phi_2$ as above and let $\tau: \Omega \to \mathbb{R}_+$ be an $(\mathcal{F}_t)$-stopping time such that $\mathbb{P}(\tau \leq T) = 1$ for some $T > 0$. By \cite[Lemma 2.3.9]{RSPDE}, we obtain
$$\mathbb{E}\bigg[\int_{0}^{\tau}\phi_1(s)\text{d}W^1(s)\cdot \int_{0}^{\tau}\phi_2(s)\text{d}W^2(s)\bigg] = \mathbb{E}\bigg[\int_{0}^{T}\mathds{1}_{]0,\tau]}\phi_1(s)\text{d}W^1(s)\cdot \int_{0}^{T}\mathds{1}_{]0,\tau]}\phi_2(s)\text{d}W^2(s)\bigg]$$ and clearly $\mathds{1}_{]0,\tau]}\phi_k \in \Lambda^2_T(W^k,U,\mathbb{R},\mathcal{P}_T)$ for $k \in \{1,2\}$. By the construction of the stochastic integral (see Proposition \ref{Isometry}), it is sufficient to prove 
\begin{equation}\label{aiai}
\mathbb{E}\bigg[\int_{0}^{T}\Phi_1(s)\text{d}W^1(s)\cdot \int_{0}^{T}\Phi_2(s)\text{d}W^2(s)\bigg] = 0
\end{equation}for all $\Phi_k \in \mathcal{E}_T(U,\mathbb{R})$. To this end let $0 = t_0 < ... < t_N = T$ be a finite partition of $[0,T]$ and set $\Phi_k := \sum_{l=0}^{N-1}B^k_l\mathds{1}_{]t_{l},t_{l+1}]}$ for $k \in \{1,2\}$. Recall that each $B_l^k$ is a map from $\Omega$ to $\text{L}(U,\mathbb{R})$, which takes finitely many values $\{\beta^k_{l_1},...,\beta^k_{l_{K_k}}\}$ and is $\mathcal{F}_{t_l}$-measurable. We may assume that $\Phi_1$ and $\Phi_2$ have the same partition. Then 
\begin{align*}
&\mathbb{E}\bigg[\int_{0}^{T}\Phi_1(s)\text{d}W^1(s)\cdot \int_{0}^{T}\Phi_2(s)\text{d}W^2(s)\bigg]= 	\sum_{l,m=0}^{N-1}\mathbb{E}\bigg[B_l^1(W^1_{t_{l+1}\wedge T}-W^1_{t_l \wedge T})\cdot B_m^2(W^2_{t_{m+1}\wedge T}-W^2_{t_m \wedge T})\bigg]  \\&=
\sum_{l,m=0}^{N-1}\sum_{i=1}^{K_1}\sum_{j=1}^{K_2}\mathbb{E}\bigg[\mathbb{E}\big[\mathds{1}_{B^1_l = \beta_{l_i}^1}\mathds{1}_{B^2_m = \beta_{m_j}^2}\beta_{l_i}^1(W^1_{t_{l+1}\wedge T}-W^1_{t_l \wedge T})\cdot \beta_{m_j}^2(W^2_{t_{m+1}\wedge T}-W^2_{t_m \wedge T})\big{|}\mathcal{F}_{t_l \vee t_m}\big]\bigg]\\&=	\sum_{l,m=0, l \neq m}^{N-1}\sum_{i=1}^{K_1}\sum_{j=1}^{K_2}\mathbb{E}\bigg[\underbrace{\mathbb{E}[\beta_{l_i}^1(W^1_{t_{l+1}\wedge T}-W^1_{t_l \wedge T})]}_{=\beta_{l_i}^1(\mathbb{E}[W^1_{t_{l+1}\wedge T}-W^1_{t_l \wedge T}])=0}\mathbb{E}\big[\mathds{1}_{B^1_l = \beta_{l_i}^1}\mathds{1}_{B^2_m = \beta_{m_j}^2}\cdot \beta_{m_j}^2(W^2_{t_{m+1}\wedge T}-W^2_{t_m \wedge T})\big{|}\mathcal{F}_{t_l \vee t_m}\big]\bigg] \\&+\sum_{l=0}^{N-1}\sum_{i=1}^{K_1}\sum_{j=1}^{K_2}\mathbb{E}\bigg[\mathbb{E}\big[\mathds{1}_{B^1_l = \beta_{l_i}^1}\mathds{1}_{B^2_l = \beta_{l_j}^2}\beta_{l_i}^1(W^1_{t_{l+1}\wedge T}-W^1_{t_l \wedge T})]\cdot \beta_{l_j}^2(W^2_{t_{l+1}\wedge T}-W^2_{t_l \wedge T})\big{|}\mathcal{F}_{t_l }\big]\bigg]\\&= 
\sum_{l=0}^{N-1}\sum_{i=1}^{K_1}\sum_{j=1}^{K_2}\mathbb{E}\bigg[\underbrace{\mathbb{E}[\beta_{l_i}^1(W^1_{t_{l+1}\wedge T}-W^1_{t_l \wedge T})\cdot \beta_{l_j}^2(W^2_{t_{l+1}\wedge T}-W^2_{t_l \wedge T})]}_{=\mathbb{E}[\beta_{l_i}^1(W^1_{t_{l+1}\wedge T}-W^1_{t_l \wedge T})]\cdot \mathbb{E}[\beta_{l_j}^2(W^2_{t_{l+1}\wedge T}-W^2_{t_l \wedge T})] = 0}\cdot \mathbb{E}\big[\mathds{1}_{B^1_l = \beta_{l_i}^1}\mathds{1}_{B^2_l = \beta_{l_j}^2}\big{|}\mathcal{F}_{t_l}\big]\bigg] = 0.
\end{align*}For the third equality we used that $W^1$ and $W^2$ are $(\mathcal{F}_t)$-Wiener processes and the $\mathcal{F}_{t_l \vee t_m}$-measurability of $\mathds{1}_{B_l^1=\beta_{l_i}^1}$ and $\mathds{1}_{B^2_m = \beta_{m_j}^2}$ and assumed (w.l.o.g.; else reverse the roles) $t_l > t_m$. In the fourth equality we once more used $\{B^1_l = \beta_{l_i}^1\}$, $\{B^2_l = \beta_{l_j}^2\} \in \mathcal{F}_{t_l}$ and the independence of $\beta_{l_i}^1(W^1_{t_{l+1}\wedge T}-W^1_{t_l \wedge T})\cdot \beta_{l_j}^2(W^2_{t_{l+1}\wedge T}-W^2_{t_l \wedge T})$ from $\mathcal{F}_{t_l}$, which follows from the independence of $W^1$ and $W^2$. This gives (\ref{aiai}). \qed \\ \\Finally, we elaborate the conclusion of the proof of Theorem \ref{UL impl JUL} in detail. Consider the situation of the final step of the proof.\\ \\
\textbf{Conclusion of proof of Theorem \ref{UL impl JUL}}: Consider a bounded $\mathcal{B}(\mathbb{B})\otimes \mathcal{B}(\mathbb{W}_0)$-measurable function $F: \mathbb{B}\times \mathbb{W}_0 \to \mathbb{R}$, which is continuous with respect to the topology of pointwise convergence in $\mathbb{W}_0$. We show that for every such $F$ the integral $\int_{\tilde{\Omega}}F(\tilde{X},\bar{\tilde{W}})\text{d}\tilde{\mathbb{P}}(\tilde{\omega})$ only depends on the distribution of $\tilde{X}$ under $\tilde{\mathbb{P}}$. Indeed, we can calculate as follows: 
\begin{align*}
&\int_{\tilde{\Omega}}F(\tilde{X},\bar{\tilde{W}})\text{d}\tilde{\mathbb{P}}(\tilde{\omega}) = \int_{\tilde{\Omega}}F\big(\tilde{X},G(\tilde{X})+\int_{0}^{\cdot}J\psi(s,\tilde{X})J^{-1}\text{d}\bar{V}^2_s\big)\text{d}\tilde{\mathbb{P}}(\tilde{\omega}) \\&= \int_{\tilde{\Omega}}F\bigg(\tilde{X},G(\tilde{X})+\sum_{l=1}^{\infty}\sum_{k=1}^{\infty}\int_{0}^{\cdot}\langle J\psi(s,\tilde{X})J^{-1}(Je_k),\bar{u}_l \rangle_{\bar{U}}\text{d}\beta_k(s)\cdot \bar{u}_l\bigg)\text{d}\tilde{\mathbb{P}}(\tilde{\omega}) \\& = \underset{n \to \infty}{\text{lim}}\underset{m \to \infty}{\text{lim}}\int_{\tilde{\Omega}}F\bigg(\tilde{X},G(\tilde{X})+\sum_{l=1}^{n}\sum_{k=1}^{m}\bigg[\underset{j \to \infty}{\text{lim}}\sum_{i=0}^{N_j-1}\langle J \psi(t_i,\tilde{X})e_k,\bar{u}_l\rangle_{\bar{U}}\big(\beta_k(t_{i+1}\wedge \cdot)-\beta_k(t_i \wedge \cdot)\big)\bigg] \bar{u}_l\bigg)\text{d}\tilde{\mathbb{P}}(\tilde{\omega}) \\& = \underset{n \to \infty}{\text{lim}}\underset{m \to \infty}{\text{lim}}\underset{j \to \infty}{\text{lim}}\underbrace{\int_{\tilde{\Omega}}F\bigg(\tilde{X},G(\tilde{X})+\sum_{l=1}^{n}\sum_{k=1}^{m}\sum_{i=0}^{N_j-1}\langle J \psi(t_i,\tilde{X})e_k,\bar{u}_l\rangle_{\bar{U}}\big(\beta_k(t_{i+1}\wedge \cdot)-\beta_k(t_i \wedge \cdot)\big) \bar{u}_l\bigg)\text{d}\tilde{\mathbb{P}}(\tilde{\omega})}_{=: F(n,m,j)}.
\end{align*}All limits are understood in the sense of pointwise convergence in $ t\geq 0$ for fixed $\tilde{\omega} \in \tilde{\Omega}$ taken out of a set of full $\tilde{\mathbb{P}}$-measure. For an orthonormal basis $(\bar{u}_l)_{l \in \mathbb{N}}$ of $\bar{U}$, the second equality follows directly from Proposition 2.4.5. in \cite{RSPDE} and $\bar{V}^2 = \sum_{k=1}^{\infty}Je_k \beta_k$ $\tilde{\mathbb{P}}$-a.s. (where $(e_k)_{k \in \mathbb{N}}$ denotes the fixed orthonormal basis of $U$ and $(\beta_k)_{k \in \mathbb{N}}$ is a family of independent, real-valued $(\tilde{\mathcal{F}}_t)$-Brownian motions on $\tilde{\Omega}$). The third equality holds due to \cite[Remark 2.8.7]{Krylov} for suitable $(N_j)_{j \in \mathbb{N}}$ and an increasing sequence $(t_i)_{i \in \mathbb{N}} \subseteq \mathbb{R}_+$. All limits can be interchanged with $F$ due to the continuity of $F$ in the aforementioned sense and can be taken out of the integral, since $F$ is bounded. For $(n, m, j) \in \mathbb{N}^3$ we continue, using Proposition 2.2.2. of \cite{RSPDE} for the second equality (note $\beta_k(t) = \langle\bar{V}^2(t),Je_k \rangle_{\bar{U}}$ and recall the independence of $\tilde{X}$ and $\bar{V}^2$):
\begin{align*}
&F(n,m,j) = \int_{\tilde{\Omega}}\mathbb{E}\bigg[F\bigg(\tilde{X},G(\tilde{X})+\sum_{l=1}^{n}\sum_{k=1}^{m}\sum_{i=0}^{N_j-1}\langle J \psi(t_i,\tilde{X})e_k,\bar{u}_l\rangle_{\bar{U}}\big(\beta_k(t_{i+1}\wedge \cdot)-\beta_k(t_i \wedge \cdot)\big) \bar{u}_l\bigg)\bigg{|}\sigma(\tilde{X})\bigg]\text{d}\tilde{\mathbb{P}}(\tilde{\omega}) \\&= \int_{\tilde{\Omega}}\mathbb{E}\bigg[F\bigg(\tilde{X}(\tilde{\omega}),G(\tilde{X}(\tilde{\omega}))+\sum_{l=1}^{n}\sum_{k=1}^{m}\sum_{i=0}^{N_j-1}\langle J \psi(t_i,\tilde{X}(\tilde{\omega}))e_k,\bar{u}_l\rangle_{\bar{U}}\big(\beta_k(t_{i+1}\wedge \cdot)-\beta_k(t_i \wedge \cdot)\big) \bar{u}_l\bigg)\bigg]\text{d}\tilde{\mathbb{P}}(\tilde{\omega}) \\&= \int_{\mathbb{B}}\mathbb{E}\bigg[F\bigg(y,G(y)+\sum_{l=1}^{n}\sum_{k=1}^{m}\sum_{i=0}^{N_j-1}\langle J \psi(t_i,y)e_k,\bar{u}_l\rangle_{\bar{U}}\big(\beta_k(t_{i+1}\wedge \cdot)-\beta_k(t_i \wedge \cdot)\big) \bar{u}_l\bigg)\bigg]\text{d}\tilde{\mathbb{P}}\circ\tilde{X}^{-1}(y).
\end{align*} Finally, we rewrite the last term on the right-hand side as
\begin{align*}
\int_{\mathbb{B}}\int_{\mathbb{W}_0}F\bigg(y,G(y)+\sum_{l=1}^{n}\sum_{k=1}^{m}\sum_{i=0}^{N_j-1}\langle J \psi(t_i,y)e_k,\bar{u}_l\rangle_{\bar{U}}\big(\langle w_{t_{i+1}\wedge \cdot},Je_k \rangle_{\bar{U}}-\langle w_{t_i\wedge \cdot},Je_k \rangle_{\bar{U}}\big) \bar{u}_l\bigg)\text{d}\tilde{\mathbb{P}}_{\bar{V}^2}(w) \text{d}\tilde{\mathbb{P}}_{\tilde{X}}(y).
\end{align*}Since $\bar{V}^2$ is a $\bar{Q}$-Wiener process, each $F(n,m,j)$ only depends on the distribution of $\tilde{X}$ under $\tilde{\mathbb{P}}$, which yields this also for $\int_{\tilde{\Omega}}F(y,w)\text{d}\tilde{\mathbb{P}}\circ (\tilde{X},\bar{\tilde{W}})^{-1}(y,w).$ Since the set of all integrals over such $F$ determines  a measure on $\mathcal{B}(\mathbb{B})\otimes \mathcal{B}(\mathbb{W}_0)$ uniquely, the joint distribution of $\tilde{X}$ and $\bar{\tilde{W}}$ under $\tilde{\mathbb{P}}$ only depends on $\tilde{\mathbb{P}}\circ \tilde{X}^{-1}$.
\renewcommand{\thesection}{B}
\section{The stochastic integral for Hilbert space-valued martingales}
	In this section we briefly recall the construction of the stochastic integral with respect to continuous, square-integrable Hilbert space-valued martingales as integrators and state its most important properties. Most parts of this section are standard and can be found in Section 3.4 of \cite{DaPratoZab} and Section 14 in \cite{MetivierStochInt}.\\
	\\ Let $(\Omega, \mathcal{F}, (\mathcal{F}_t)_{t \in [0,T]}, \mathbb{P})$ be a stochastic basis and $U$ a separable Hilbert space with an orthonormal basis $\{e_n\}_{n \in \mathbb{N}}$. We introduce the Banach space
	$$\mathcal{M}_c^2(T;U) := \{M:\Omega \to C([0,T];U)\big{|}M \text{ continuous }(\mathcal{F}_t)\text{-martingale, } ||M||_{\mathcal{M}^2_T} < \infty\},$$where the norm $||\cdot||_{\mathcal{M}^2_T}$ on $\mathcal{M}_c^2(T;U)$ is defined by 
	$||M||_{\mathcal{M}^2_T} := (\mathbb{E}[||M_T||^2_{U}])^{\frac{1}{2}} = \underset{t \in [0,T]}{\text{sup}}(\mathbb{E}[||M_t||^2_{U}])^{\frac{1}{2}}.$ By the maximal inequality, $||\cdot||_{\mathcal{M}^2_T}$ is equivalent to the $L^2(\Omega, \mathcal{F},\mathbb{P}; L^{\infty}([0,T];U))$-norm on $\mathcal{M}^2_c(T;U)$.
	\subsubsection*{The quadratic variation of a Hilbert space-valued, square-integrable continuous martingale}
	It is well-known that for $M \in \mathcal{M}^2_c(T;U)$ (with $M_0 = 0)$ there exists a unique real-valued, increasing, $(\mathcal{F}_t)$-adapted, continuous process $(\langle M\rangle_t)_{t \in [0,T]}$ (with $\langle M \rangle _0 \equiv 0$) such that $||M_t||^2-\langle M\rangle_t$ is a continuous $(\mathcal{F}_t)$-martingale. Let $\alpha_M := \mathbb{P}(\text{d}\omega)\otimes \langle M \rangle(\omega,\text{d}t)$ as a measure on $\mathcal{B}([0,T])\otimes \mathcal{F}$. Now we define the \textit{quadratic variation} of $M$.
	\begin{dfn}\label{testtest}The $\text{L}_1(U)$-valued process $(\ll M \gg_t)_{t \in [0,T]}$, defined through $$ \ll M \gg_t \,\,= \sum_{i,j=1}^{\infty}\langle \langle M_i, M_j \rangle \rangle_t  e_i \otimes e_j$$ in $\text{L}_1(U)$ is the \textit{quadratic variation (process)} of $M$.
		Here $M_i$ denotes the real-valued martingale $\langle M, e_i \rangle_U,\,\,\,t \in [0,T]$ and we set $e_i\otimes e_j(u) := e_i \langle e_j, u\rangle_U$, where $\langle \langle \cdot , \cdot \rangle \rangle$ denotes quadratic covariation for real-valued martingales.
	\end{dfn}
	\begin{prop}\label{thisthat}
		There exists a (up to an $\alpha_M$-zero set) unique predictable process $Q_M:[0,T ]\times \Omega \to \text{L}^+_1(U)$ such that 
		$$\ll M \gg _t \,\,= \int_{[0,t]}Q_M(s)\text{d}\langle M\rangle_s.$$
		The integral above is a pathwise Bochner-integral, taking values in the separable Banach space $\text{L}_1(U).$
	\end{prop}
	
	\begin{dfn}
		An $\text{L}(U)$-valued process $(B_t)_{t \in [0,T]}$ with $B_t$ non-negative for every $t$ is called \textit{increasing}, if for every $0\leq s\leq t \leq T$ and $\omega \in \Omega$ the operator $B_t(\omega)-B_s(\omega)$ is non-negative.
	\end{dfn}
	\begin{prop}\label{Prop Qu Var}
		An $\text{L}_1(U)$-valued process $V$ is the quadratic variation of $M \in \mathcal{M}_c^2(T;U)$ with $M_0=0$ if and only if it is increasing, continuous, $(\mathcal{F}_t)$-adapted with $V_0=0$ and such that the process  
		$$\langle M_t,a \rangle_{U} \langle M_t,b \rangle_{U} - \langle V_t a,b \rangle_{U}, \, \, \, t \in [0,T]$$
		is an $\mathbb{R}$-valued $(\mathcal{F}_t)$-martingale for all $a,b \in U.$
	\end{prop}Next we define the notion of the \textit{quadratic cross variation} for two Hilbert space-valued martingales and draw a connection to the quadratic variation reminiscent to the real-valued case.
	\begin{dfn}\label{Def qu cross var}
		Let $M,N \in \mathcal{M}^2_c(T;U)$. The \textit{quadratic cross variation} of $M$ and $N$ is defined through
		$\ll M,N \gg_t \,\,:= \sum_{i,j=1}^{\infty}\langle \langle M_i,N_j \rangle \rangle_t e_i \otimes e_j, \,\,\,t \in [0,T].$
	\end{dfn}
	\begin{lem}
		For $M,N \in \mathcal{M}^2_c(T;U)$ the following formula holds $\mathbb{P}$-a.s. for every $t \in [0,T]$:
		\begin{equation*}
		\ll M+N \gg_t \,=\, \ll M \gg_t + \ll N \gg_t + \ll M,N \gg_t + \ll N,M \gg_t.
		\end{equation*}
	\end{lem}
	\begin{proof} The claim follows immediately by Definitions \ref{testtest}, \ref{Def qu cross var} and the bilinearity of the cross variation.
	\end{proof}
	\begin{kor}\label{Quad Var additive for independent}
		\begin{enumerate}
			\item[(i)] Let $M, N \in \mathcal{M}^2_c(T;U)$ be such that the real-valued continuous martingales $\langle M, e_i \rangle_U$ and $\langle N,e_j\rangle_U $ have covariation zero for every $i,j \in \mathbb{N}.$ Then we have for every $t \in [0,T]$
			\begin{equation}\label{onenine}
			\ll M+N \gg _t \,=\, \ll M \gg _t + \ll N \gg _t \, \, \,\mathbb{P}\text{-a.s.}
			\end{equation}
			\item [(ii)] 	In particular (\ref{onenine}) holds, if $M$ and $N$ are independent.
		\end{enumerate}
	\end{kor}
	Finally, we state the Hilbert space-version of Lévy's characterization of Brownian motion, which is used in the proof of Theorem \ref{UL impl JUL}.
	\begin{prop}\label{Generalized Levy}(Generalized Lévy-Characterization)\\
		Let $M \in \mathcal{M}^2_c(T;U)$ be such that $M_0 = 0 \,\,\mathbb{P}$-a.s and let $Q \in \text{L}^+_1(U)$. Then the following are equivalent.
		\begin{enumerate}
			\item [(i)] $(M_t)_{t \in [0,T]}$ is an $(\mathcal{F}_t)$-$Q$-Wiener process on $(\Omega, \mathcal{F}, (\mathcal{F}_t)_{t \in [0,T]}, \mathbb{P})$ (in particular $M_t-M_s$ is independent of $\mathcal{F}_s$ for all $0 \leq s < t \leq T$).
			\item[(ii)] $\ll M \gg_t \,= tQ$ $\mathbb{P}$-a.s. for $t \in [0,T]$.
		\end{enumerate}
	\end{prop}
	\subsubsection*{The construction of the stochastic integral}\label{The Construction of the Stochastic Integral}We continue with the construction of the stochastic integral.
	\begin{rem}(c.f. \cite[Prop. 2.3.4.]{RSPDE}) If $Q: [0,T]\times \Omega \to \text{L}^+_1(U)$, then there exists a unique, operator-valued process $Q^{\frac{1}{2}}:  [0,T]\times \Omega \to \text{L}_2(U)$ such that $Q^{\frac{1}{2}}(t,\omega) \circ Q^{\frac{1}{2}}(t,\omega) = Q(t,\omega)$ for all $(t,\omega) \in [0,T]\times \Omega$.  
	\end{rem}The following construction and results are standard. One starts with the construction of stochastic integral with respect to \textit{elementary integrands} and then extends this definition through a suitable isometry. In the sequel $H$ denotes another separable, (infinite-dimensional) Hilbert space.
	\begin{dfn}\label{Def simple proc}
		A process $A: [0,T]\times \Omega \to \text{L}(U,H)$ is \textit{elementary}, if it is of the form
		$$A(t,\omega) = \sum_{k=0}^{N-1}\phi_k(\omega)\mathds{1}_{]t_k,t_{k+1}]}(t),$$where $\phi_k: \Omega \to \text{L}(U,H)$ has finite image in $\text{L}(U,H)$ and is $\mathcal{F}_{t_k}$-measurable with respect to the strong Borel $\sigma$-algebra for every $k \in \{0,...,N-1\}$ and $ 0 = t_0 < t_1 < ... < t_N = T$ is a finite partition of $[0,T]$. The set of all such processes is denoted by $\mathcal{E}_T(U,H)$.
	\end{dfn}
	\begin{dfn}
		For $A = \sum_{k=0}^{N-1}\phi_k\mathds{1}_{]t_k,t_{k+1}]} \in \mathcal{E}_T(U,H)$ and $M \in \mathcal{M}^2_c(T;U)$ the \textit{stochastic integral of }$A$ \textit{with respect to }$M$ is defined through
		$$\int_{0}^{t}A(s)\text{d}M_s := \sum_{k=0}^{N-1}\phi_k\big(M_{t_{k+1}\wedge t}-M_{t_k\wedge t}\big), \,\,\,t\in [0,T].$$
	\end{dfn}
	\begin{prop}
		Let $A \in \mathcal{E}_T(U,H)$. Then the stochastic integral process $\big(\int_{0}^{t}A(s)\text{d}M_s\big)_{t \in [0,T]}$ is an element of $\mathcal{M}^2_c(T;H).$
	\end{prop}Now we want to extend this definition through a suitable isometry. We need the following space of operator-valued processes. In the sequel we abbreviate the Hilbert-Schmidt norm by $||\cdot||_2$ when no confusion is possible.
	\begin{dfn}\label{Lambda space}
		Let $M \in \mathcal{M}^2_c(T;U)$ and $Q_M$ as in Proposition \ref{thisthat}. The vector space $\Lambda^2_T(M,U,H,\mathcal{P}_T)$ is defined by containing processes
		$
		X: [0,T]\times \Omega \to \text{Lin}(U,H) ,$ which fulfill
		\begin{enumerate}
			\item [(i)] $\mathcal{D}(X(t,\omega)$) $\supseteq Q_M^{\frac{1}{2}}(t,\omega)(U)$ $\text{ for all } (t,\omega) \in [0,T]\times \Omega$.
			\item[(ii)] For every $u \in U$ the process $X \circ Q_M^{\frac{1}{2}}(u): [0,T]\times \Omega \to H$ is $(\mathcal{F}_t)$-predictable.
			\item[(iii)] $\int_{[0,T]\times \Omega}||X \circ Q_M^{\frac{1}{2}}||^2_2\text{d}\alpha_M < +\infty.$
		\end{enumerate}
	\end{dfn}
	\begin{prop}The bilinear form
		$(X,Y) \mapsto \int_{[0,T]\times \Omega}\text{tr}\big[(X \circ Q_M^{\frac{1}{2}}) (Y \circ Q_M^{\frac{1}{2}})^*\big]\text{d}\alpha_M$ is a scalar product on $\Lambda^2_T(M,U,H,\mathcal{P}_T)$. Equipped with this scalar product, $\Lambda^2_T(M,U,H,\mathcal{P}_T)$ is a Hilbert space. In particular, denoting the corresponding norm by $||\cdot ||_{\Lambda^2_T}$, we have $||X||_{\Lambda^2_T} = \int_{[0,T]\times \Omega}||X \circ Q_M^{\frac{1}{2}}||^2_{2}\text{d}\alpha_M$ for $X \in \Lambda^2_T(M,U,H,\mathcal{P}_T)$.
	\end{prop}
	For every element of $\Lambda^2_T(M,U,H,\mathcal{P}_T)$ the stochastic integral with respect to $M$ can be defined. This is contained in the following two statements.
	\begin{prop}\label{important}
		$\Lambda^2_T(M,U,H,\mathcal{P}_T)$ is the closure of $\mathcal{E}_T(U,H)$ with respect to the norm $||\cdot||_{\Lambda^2_T}$.
	\end{prop}
	\begin{prop}\label{Isometry}
		Let $M \in \mathcal{M}_c^2(T;U).$ There exists a unique linear isometric map from ($\Lambda^2_T(M,U,H,\mathcal{P}_T)$, $||\cdot||_{\Lambda^2_T})$ to $(\mathcal{M}^2_c(T;H)$, $||\cdot||_{\mathcal{M}^2_T})$, which extends the linear map $\Phi_M: \mathcal{E}_T(U,H) \to \mathcal{M}_c^2(T;H)$, defined through
		$$\Phi_M(A):= \bigg(\sum_{k=0}^{N-1}\phi_k\big(M_{t_{k+1}\wedge t}-M_{t_k\wedge t}\big)\bigg)_{t \in [0,T]}$$ for $A := \sum_{k=0}^{N-1}\phi_k\mathds{1}_{]t_k,t_{k+1}]} \in \mathcal{E}_T(U,H).$ For $A \in \Lambda^2_T(M,U,H,\mathcal{P})$ the continuous, $(\mathcal{F}_t)$-adapted, square-integrable $H$-valued process $\Phi_M(A)$ is called \textit{stochastic integral (of $A$ with respect to $M$)} and is denoted by $(\int_{0}^{t}A(s)\text{d}M_s)_{t \in [0,T]}$ or simply by $A.M.$
	\end{prop}The final step of the construction consists of a localization in order to enlarge the class of admissible integrands. Let $M$ be as before and consider an operator-valued process $A$, which fulfills (i) and (ii) of Definition \ref{Lambda space}, but instead of (iii) we now only require $A$ to fulfill $$\mathbb{P}\bigg(\int_{0}^{T}||X \circ Q^{\frac{1}{2}}_M(s)||^2_2\text{d}\langle M \rangle(s)\bigg) < +\infty. $$We denote the set of all such $A$ by $\Lambda_T(M,U,H,\mathcal{P}_T)$. Clearly $\Lambda^2_T(M,U,H,\mathcal{P}_T) \subseteq \Lambda_T(M,U,H,\mathcal{P}_T)$. Reminiscent to Step 4 of Section 2.3.2 in \cite{RSPDE}, one defines  	\begin{equation}\label{keykey}\int_{0}^{t}A(s)\text{d}M_s := \underset{n \to +\infty}{\text{lim}}\int_{0}^{t}\mathds{1}_{]0,\tau_n]}A(s)\text{d}M_s \,\,\, \mathbb{P}\text{-a.s.}
	\end{equation} for any sequence of $(\mathcal{F}_t)$-stopping times $(\tau_n)_{n \in \mathbb{N}}$, which fulfills
	\begin{enumerate}
		\item [(i)] $(\tau_n)_{n \in \mathbb{N}}$ is non-decreasing and converges to $T$ $\mathbb{P}$-a.s.,
		\item[(ii)] $\mathds{1}_{]0,\tau_n]}A \in \Lambda^2_T(M,U,H,\mathcal{P}_T)$ for every $n \in \mathbb{N}$.
	\end{enumerate} For example, one may choose $\tau_n(\omega) := \text{inf}\big\{t\in [0,T]\big{|}\int_{0}^{t}||A \circ Q^{\frac{1}{2}}_M(s,\omega)||^2_2 \text{d}\langle M \rangle (s) > n\big\} \wedge T$ and one verifies that (\ref{keykey}) does not depend on the particular sequence $(\tau_n)_{n \in \mathbb{N}}$. Clearly for $A \in \Lambda_T(M,U,H,\mathcal{P}_T)$ the stochastic integral $A.M$ is a continuous, local $H$-valued martingale.\\ \\
	Finally, we introduce the definition of stochastic integrals with respect to continuous \textit{local }martingales.
	\begin{dfn}
		Let $(M_t)_{t \in [0,T]}$ be a continuous, $(\mathcal{F}_t)$-adapted $U$-valued local martingale such that for every element $\tau_n$ of its localizing sequence $(\tau_n)_{n \in \mathbb{N}}$, the martingale $(M_{t \wedge \tau_n})_{t\in [0,T]}$ belongs to $\mathcal{M}^2_c(T;U)$. Define $$\Lambda^2_{T,\text{loc}}(M,U,H,\mathcal{P}_T) := \underset{n \in \mathbb{N}}{\bigcap}\Lambda^2_T(M_{\cdot \wedge \tau_n},U,H,\mathcal{P}_T)$$and for $A \in \Lambda^2_{T,\text{loc}}(M,U,H,\mathcal{P}_T)$ set $\int_{0}^{t}A(s)\text{d}M_s := \underset{n \to +\infty}{\text{lim}}\int_{0}^{t}A(s)\text{d}M_{s \wedge \tau_n},\,\,\,t\in [0,T]$. This definition does not depend on the sequence $(\tau_n)_{n \in \mathbb{N}}$.
	\end{dfn}
	\subsubsection*{Properties of the stochastic integral}
	The following proposition and its proof can be found in Section 4.3 of \cite{DaPratoZab}. 
	\begin{prop}\label{Shape qu var, Q, a for stoch int}
		Let $Q \in \text{L}^+_1(U)$, $W \in \mathcal{M}^2_c(T;U)$ be a $(\mathcal{F}_t)$-$Q$-Wiener process and $A \in \Lambda^2_T(W,U,H,\mathcal{P}).$ Then \begin{equation*}\label{oneten}
		\ll A.W \gg_t = \int_{0}^{t}(A(s)\circ Q^{\frac{1}{2}})(A(s)\circ Q^{\frac{1}{2}})^*\text{d}s,\,\,\,t \in [0,T] \,\,\mathbb{P}\text{-a.s.}
		\end{equation*}
	\end{prop}It is well known that $H_1.(H_2.M) = (H_1 \cdot H_2).M$ holds in the case of finite-dimensional stochastic integration. We use the Hilbert space-analogue of this result, stated in Proposition \ref{Asso stoch int} below, multiple times within our main proofs. This proposition and Lemma \ref{super lemma} are taken from \cite{Asperti} (c.f. Lemma 3.6. and Theorem 3.7. therein). We also need two slight generalizations of this result, which we both state and prove in Proposition \ref{localize} at the end of this appendix. Let $G$ denote another separable, infinite-dimensional Hilbert space.
	\begin{lem}\label{super lemma}
		Let $M \in \mathcal{M}^2_c(T;U)$, $A \in \Lambda^2_T(M,U,H,\mathcal{P}_T)$ and $B: [0,T]\times \Omega \to \text{Lin}(H,G)$. The following are equivalent:
		$$\text{(i) } B \circ A \in \Lambda^2_T(M,U,G,\mathcal{P}_T) \,\,\,\,\,\,\,\,\,\,
		\text{(ii) } B \in \Lambda^2_T(A.M,H,G,\mathcal{P}_T).$$
		In this case $B.(A.M)$ and $B\circ A.M$ are equal in norm in $\mathcal{M}^2_c(T;G)$.
	\end{lem}From here we can readily obtain the following important statement:
	\begin{prop}\label{Asso stoch int}
		Let $M, A$ and $B$ be as in the previous lemma such that the equivalent properties therein are fulfilled. Then $\big(\int_{0}^{t}B \circ (A)(s)\text{d}M_s\big)_{t \in [0,T]}$ and $\big(\int_{0}^{t}B(s)\text{d}(\int_{0}^{s}A(r)\text{d}M_r)\big)_{t \in [0,T]}$ are equal in $\mathcal{M}^2_c(T;G)$. In particular we have 
		$$\int_{0}^{t}B \circ (A)(s)\text{d}M_s =  \int_{0}^{t}B(s)\text{d}\bigg(\int_{0}^{s}A(r)\text{d}M_r\bigg),\,\,\, t \in [0,T] \, \,\mathbb{P}\text{-a.s.}$$
	\end{prop}Finally, we generalize the above proposition to elements $A \in \Lambda_T(M,U,H,\mathcal{P}_T)$ and $B \in \Lambda_T(A.M,H,G,\mathcal{P}_T)$.
	
	\begin{prop}\label{localize}Let $W$ be an $(\mathcal{F}_t)$-$Q$-Wiener 		 process for $Q \in \text{L}^+_1(U)$.
		\begin{enumerate} 
			\item [(i)] Let $A \in \Lambda^2_T(W,U,H,\mathcal{P}_T)$ and $B: [0,T]\times \Omega \to \text{Lin(H,G)}$ such that $B \circ A \in \Lambda_T(W,U,G,\mathcal{P}_T)$. Then $B \in \Lambda_T(A.W,H,G,\mathcal{P}_T)$ and $$\int_{0}^{t}B(s)\text{d}\bigg(\int_{0}^{s}A(r)\text{d}W_r\bigg)=\int_{0}^{t}B\circ A(s)\text{d}W_s,\,\,\, t \in [0,T] \,\, \mathbb{P}\text{-a.s.}$$ 
			\item[(ii)] Let $A \in \Lambda_T(W,U,H,\mathcal{P}_T)$ and $B: [0,T]\times \Omega\to \text{Lin}(H,G)$ such that $B \circ A \in \Lambda^2_T(W,U,G,\mathcal{P}_T)$. Then $B \in \Lambda^2_{T,\text{loc}}(A.W,H,G,\mathcal{P}_T)$ and 
			$$\int_{0}^{t}B(s)\text{d}\bigg(\int_{0}^{s}A(r)\text{d}W_r\bigg)=\int_{0}^{t}B\circ A(s)\text{d}W_s,\,\,\, t \in [0,T] \,\, \mathbb{P}\text{-a.s.}$$  
		\end{enumerate}
		
	\end{prop}
	\begin{proof}
		\begin{enumerate}
			\item [(i)] Since $B\circ A\in \Lambda_T(W,U,G,\mathcal{P}_T)$, there exists a sequence of $(\mathcal{F}_t)$-stopping times $(\tau_n)_{n \in \mathbb{N}}$ with properties (i) and (ii), mentioned in the localization step of the construction on the previous pages, such that
			$$\int_{0}^{t}B\circ A(s)\text{d}W_s = \underset{n \to +\infty}{\text{lim}}\int_{0}^{t}\mathds{1}_{]0,\tau_n]}B\circ A(s)\text{d}W_s,\,\,\, t \in [0,T]\,\,\mathbb{P}\text{-a.s.}$$Using Lemma \ref{super lemma} and Proposition \ref{Asso stoch int} we obtain $B \in \Lambda_T(A.W,H,G,\mathcal{P}_T)$ and that $(\tau_n)_{n \in \mathbb{N}}$ is also a proper localizing sequence for $B$. Therefore $$\int_{0}^{t}\mathds{1}_{]0,\tau_n]}B\circ A(s)\text{d}W_s = \int_{0}^{t}\mathds{1}_{]0,\tau_n]}B(s)\text{d}\bigg(\int_{0}^{s}A(r)\text{d}W_r\bigg),\,\,\,t \in [0,T]\,\,\mathbb{P}\text{-a.s.}$$for every $n \in \mathbb{N}$. But since by definition $$\underset{n \to +\infty}{\text{lim}}\int_{0}^{t}\mathds{1}_{]0,\tau_n]}B(s)\text{d}\bigg(\int_{0}^{s}A(r)\text{d}W_r\bigg) = \int_{0}^{t}B(s)\text{d}\bigg(\int_{0}^{s}A(r)\text{d}W_r\bigg)$$for every $t \in [0,T]$ $\mathbb{P}$-a.s., the assertion follows.
			\item[(ii)] Since $A \in \Lambda_T(M,U,H,\mathcal{P}_T)$, there exists a sequence $(\sigma_n)_{n \in \mathbb{N}}$ of $(\mathcal{F}_t)$-stopping times with properties (i) and (ii) as above such that
			$$\int_{0}^{t}A(s)\text{d}W_s = \underset{n \to +\infty }{\text{lim}}\int_{0}^{t}\mathds{1}_{]0,\sigma_n]}A(s)\text{d}W_s,\,\,\,t\in [0,T]\,\,\mathbb{P}\text{-a.s.}$$ and $\mathds{1}_{]0,\sigma_n]}A \in \Lambda^2_T(W,U,H,\mathcal{P}_T)$ for all $n \in \mathbb{N}$. Since $B\circ A \in \Lambda^2_T(W,U,G,\mathcal{P}_T)$, we also have $$B \circ \mathds{1}_{]0,\sigma_n]}A = \mathds{1}_{]0,\sigma_n]}B\circ A \in \Lambda^2_T(W,U,G,\mathcal{P}_T).$$ Consequently we conclude that all terms in the following equation are well-defined and fulfill, for every $n \in \mathbb{N},$
			\begin{equation}\label{olaola}
			\int_{0}^{t}B(s)\text{d}\big(\int_{0}^{s}\mathds{1}_{]0,\sigma_n]}A(r)\text{d}W_r\big) = \int_{0}^{t}\mathds{1}_{]0,\sigma_n]}B\circ A(s)\text{d}W_s,\,\,\,t \in [0,T]\,\,\mathbb{P}\text{-a.s.}
			\end{equation}For $n\to +\infty$, the right-hand side of (\ref{olaola}) clearly converges $\mathbb{P}$-a.s. to $\int_{0}^{t}B\circ A(s)\text{d}W_s$ with $\mathbb{P}$-zero set independent of $t \in [0,T]$, while the limit of the left-hand side is by definition equal to $\int_{0}^{t}B(s)\text{d}\big(\int_{0}^{s}A(r)\text{d}W_r\big)$, again with zero set independent of $t \in [0,T]$. This concludes the proof.\qedhere	\end{enumerate}
	\end{proof}
	Finally, we mention that the entire construction and all properties presented in this section immediately carry over to the case $T = +\infty$. We would like to stress, however, that these extended stochastic integrals on $ \Omega \times\mathbb{R}_+$ are in general only continuous \textit{local} martingales.
\section*{Acknowledgements}
The author would like to cordially thank Prof. Michael Röckner for drawing his attention to the result of Cherny and for many fruitful discussions along the making of this paper.

\bibliography{MALiterature}

\end{document}